\documentclass[12p, reqno]{amsart}
\usepackage{amssymb, amsmath, amsthm}

\title[An Example of Rapid Evolution of Complex Cycles]{An Example of Rapid Evolution of Complex Limit Cycles}

\author[Nikolay Dimitrov]{}

\email{dimitrov@math.mcgill.ca}



\newtheorem*{Pontryagin}{Pontryagin's Theorem}
\newtheorem*{main}{Main Result}
\newtheorem{theorem}{Theorem}
\newtheorem{lemma}{Lemma}[section]
\newtheorem{proposition}{Proposition}
\newtheorem{definition}{Definition}
\newtheorem{corollary}{Corollary}[section]

\newcommand{\Integers}{\mathbb{Z}}
\newcommand{\Naturals}{\mathbb{N}}
\newcommand{\CC}{\mathbb{C}^2}
\newcommand{\cptwo}{\mathbb{CP}^2}
\newcommand{\RR}{\mathbb{R}^2}
\newcommand{\cc}{\mathbb{C}}
\newcommand{\Band}{\mathbb{B}}
\newcommand{\Disc}{\mathbb{D}}

\newcommand{\Torus}{\mathbb{T}}
\newcommand{\ThreeSphere}{\mathbb{S}^3}

\newcommand{\Azeroq}{A_0(q_0)}
\newcommand{\Azeroprimeq}{A'_0(q_0)}
\newcommand{\Aoneprimeq}{A'_1(q_0)}

\newcommand{\Ahatzero}{\hat{A}_0}
\newcommand{\Ahatzeroprime}{\hat{A}'_0}
\newcommand{\Ahatoneprime}{\hat{A}'_1}
\newcommand{\Ahat}{\hat{A}}

\newcommand{\alphaq}{\alpha_{q',\e}}

\newcommand{\Atlas}{\mathcal{A}}

\newcommand{\betaq}{\beta_{q',\e}}

\newcommand{\h}{c}

\newcommand{\Deltae}{\Delta_{\e}}

\newcommand{\deltae}{\delta_{\e}}
\newcommand{\deltahat}{\hat{\delta}}

\newcommand{\etamax}{\eta_{\text{max}}}
\newcommand{\etwostar}{\e^{**}}

\newcommand{\Gammatilde}{\tilde{\Gamma}}
\newcommand{\Gammahat}{\hat{\Gamma}}
\newcommand{\gammahat}{\hat{\gamma}}
\newcommand{\gammatilde}{\tilde{\gamma}}

\newcommand{\e}{\varepsilon}
\newcommand{\estar}{\e^{*}}

\newcommand{\CtimesS}{\cc \times S_1}

\newcommand{\Fol}{\mathcal{F}_{a,\e}}
\newcommand{\Foliation}{\mathcal{F}}
\newcommand{\Field}{F_{a,\varepsilon}}
\newcommand{\Ftilde}{\tilde{F}}
\newcommand{\Fieldzero}{F_0}
\newcommand{\Folzero}{\Foliation_0}

\newcommand{\Fole}{\mathcal{F}_{\e}}
\newcommand{\Fielde}{F_{\varepsilon}}
\newcommand{\Folehat}{\hat{\mathcal{F}}_{\e}}

\newcommand{\Leafae}{L_{a,\e}}
\newcommand{\Leafe}{L_{\e}}
\newcommand{\Leafaeq}{L_{a,\e}(q)}

\newcommand{\Leafhate}{\hat{L}_{\e}}

\newcommand{\Leafhatezp}{\hat{L}_{\e}(z,p)}

\newcommand{\Phat}{\hat{P}}
\newcommand{\Pmap}{P_{a,\e}}
\newcommand{\Pmape}{P_{\e}}
\newcommand{\Pmaphat}{\hat{P}_{a,\e}}
\newcommand{\Pmaphate}{\hat{P}_{\e}}

\newcommand{\etatilde}{\tilde{\eta}}

\newcommand{\Uq}{U_{q'}}

\newcommand{\Uhatzzero}{\hat{U}_{\hat{z}_0}}
\newcommand{\Phihatzzero}{\hat{\phi}_{\hat{z}_0,\e}}
\newcommand{\Uhatzone}{\hat{U}_{\hat{z}_1}}
\newcommand{\Phihatzone}{\hat{\phi}_{\hat{z}_1,\e}}
\newcommand{\Phihatzoneprime}{\hat{\phi}_{\hat{z}_1,\e'}}

\newcommand{\zhat}{\hat{z}}

\begin{document}

\maketitle

\centerline{\scshape Nikolay Dimitrov}
\medskip
{\footnotesize
 \centerline{Department of Mathematics and Statistics}
   \centerline{McGill University}
   \centerline{805 Sherbrooke W.}
   \centerline{Montreal, QC H3A 2K6, Canada}
 }

\begin{abstract}
In the current article we study complex cycles of higher
multiplicity in a specific polynomial family of holomorphic
foliations in the complex plane. The family in question is a
perturbation of an exact polynomial one-form giving rise to a
foliation by Riemann surfaces. In this setting, a complex cycle is
defined as a nontrivial element of the fundamental group of a leaf
from the foliation. In addition to that, we introduce the notion
of a multi-fold cycle and show that in our example there exists a
limit cycle of any multiplicity. Furthermore, such a cycle gives
rise to a one-parameter family of cycles continuously depending on
the perturbation parameter. As the parameter decreases in absolute
value, the cycles from the continuous family escape from a very
large subdomain of the complex plane.
\end{abstract}

\section{Introduction} \label{Section_Introduction}

Limit cycles of planar polynomial vector fields have long been a
focus of extensive research. For instance, one of the major
problems in this area of dynamical systems is the famous Hilbert's
16th problem \cite{I02} asking about the number and the location
of the limit cycles of a polynomial vector field of degree n in
the plane. Since the original Hilbert's problem continues to be
very persistent, some simplifications have been considered as
well. Among them is the so called infinitesimal Hilbert's 16
problem \cite{I02}, \cite{IY} concerned with the number of limit
cycles that can bifurcate from periodic solutions of a polynomial
Hamiltonian planar system by a small polynomial perturbation.
Recently, an answer to this question has been given in an article
by Binyamini, Novikov and Yakovenko \cite{BNY}.

When studying a planar polynomial vector field, an extension to
the complex domain proves to be helpful, an idea that can be
attributed to Petrovskii and Landis \cite{PL55}, \cite{PL57}. In
this way a polynomial complex vector field is obtained and the
holomorphic curves tangent to it form a partition of the complex
plane by Riemann surfaces, called a \emph{polynomial complex
foliation with singularities}, or in short \emph{polynomial
complex foliation} \cite{I02}, \cite{IY}.

Following the idea of complexification, polynomial deformations of
planar Hamiltonian vector fields could be extend to $\CC$. More
precisely one could consider the complex line field
\begin{equation}\label{GeneralFoliation}
\ker(dH + \e \omega)
\end{equation}
with a one-form $\omega = Adx + Bdy,$ where $A,B$ and $H \in
\cc[x,y]$ are polynomials with complex coefficients and $\e$ is a
small complex parameter.

For the purposes of the current study, we focus our attention on a
specific example. Let $H$ be the simple polynomial
\begin{displaymath}
H = x^2 + y^2.
\end{displaymath}
Choose polynomial one-forms $\omega_1$ and $\omega_2$ as follows:
\begin{align*}
\omega_1 = (H-1)(y dx - x dy) \,\,\,\, \text{and} \,\,\,\,
\omega_2 = y \, dH.
\end{align*}
Consider the two parameter family of complex line fields
\begin{equation} \label{ExampleFoliation}
\Field = \ker \Bigl(dH + \e(\omega_1 + a\omega_2)\Bigl),
\end{equation}
where $\e$ and $a$ are the parameters. Notice that the family is
of the form (\ref{GeneralFoliation}).

As mentioned earlier, the holomorphic curves tangent to $\Field$
form a foliation of Riemann surfaces in $\CC$ further denoted by
$\Fol(\CC)$. For example, consider the Riemann surface
$$S_1=\{(x,y) \in \CC \, | \, x^2 + y^2 =1\}.$$ As we are going
to see in the next section \ref{Section_Main_Theorem}, the surface
$S_1$ is tangent to the complex line field $\Field$ for any value
of the parameters $a$ and $\e$ so it is a leaf of $\Fol(\CC)$. Fix
the unit circle $\delta_0 = S_1 \cap \RR$. Notice, that in the
case of real $a$ and $\e$ the phase curves of
(\ref{ExampleFoliation}) restricted to $\RR$ are topologically
either lines or circles, i.e. curves with either a trivial or a
non-trivial (isomorphic to $\Integers$) fundamental group. For
example, $\delta_0$ is such a circular phase curve. This simple
observation leads us to the definition of a marked complex cycle.

\begin{definition} \label{cycle}
A marked complex cycle of a complex foliation is a nontrivial
element of the fundamental group of a leaf from the foliation with
a marked base point.
\end{definition}

\noindent We denote a marked complex cycle by $(\Delta, q)$ where
$\Delta$ is the homotopy class of loops on the leaf, all passing
through the same base point $q$. Each loop from $\Delta$ will be
called a \emph{representative} of the cycle. In general, a real
phase curve of a polynomial vector field in $\RR$ extends to a
Riemann surface tangent to the vector field's complexification in
$\CC.$ Thus, a closed phase curve in $\RR$ defines a loop on the
corresponding complex leaf, giving rise to a nontrivial element
from the fundamental group of that leaf \cite{I02}. In other
words, a real closed phase curve is a marked complex cycle on its
complexification. As an illustration, the leaf $S_1$ is the
complexification of the real trajectory $\delta_0$. The surface
$S_1$ is topologically a cylinder and $\delta_0$ is a nontrivial
loop on it. Denoting by $q_0$ the point $(1,0) \in S_1$ and by
$\Delta_0$ the homotopy class of $\delta_0$ relative to $q_0$ we
obtain a marked complex cycle $(\Delta_0, q_0)$ of $\Fol(\CC)$.

When $\e=0$ the line field (\ref{ExampleFoliation}) will be
denoted by $\Fieldzero$ and its corresponding foliation by
$\Folzero(\CC)$. From now on, we are going to refer to
$\Folzero(\CC)$ as the \emph{integrable foliation} and to
$\Fol(\CC)$ as the \emph{perturbed foliation}. Notice that
$\Folzero(\CC)$ consists of algebraic leaves of the form $S_{\h} =
\{(x,y) \in \CC \,|\, H(x,y)=\h \}$ embedded in $\CC$, where $\h
\in \cc$. All leaves with $\h \neq 0$ are topological cylinders.
Our basic approach will be to study the complex cycles of the more
complicated $\Fol(\CC)$ by taking advantage of the simplicity of
$\Folzero(\CC)$.

One of the very useful tools for converting some of the
topological properties of the foliation into dynamical properties
of a holomorphic map of complex dimension one is the so called
\emph{Poincar\'e displacement map} \cite{I02}, \cite{IY}. Next, we
present a construction of it in the case of example
(\ref{ExampleFoliation}). Let $T'$ be a complex segment (a small
disc on a complex line in $\CC$) passing through $q_0$ and
transverse to the surface $S_1$. Consider an annular neighborhood
$A(\delta_0)$ of $\delta_0$ on the surface $S_1$. Next, take a
tubular neighborhood $N(\delta_0)$ of $A(\delta_0)$ in $\CC$. It
is diffeomorphic to a direct product $A(\delta_0) \times \Disc,$
where $\Disc \subset \cc$ is the unit disc. Let $\varrho$ be the
projection of $N(\delta_0)$ onto $A(\delta_0)$ with respect to
that direct product structure. Without loss of generality, we can
think that $T'=\varrho^{-1}(q_0)$. Let $T \subset T'$ be a small
enough open neighborhood of $q_0$ in $T'$. Take any point $q \in
T$ and consider the leaf $\Leafaeq$ from the foliation $\Fol(\CC)$
that passes through $q$. Starting from $q \in \Leafaeq$, lift the
loop $\delta_0$ to the unique path on $\Leafaeq$ that covers
$\delta_0$ under the projection $\varrho$. The second end-point of
this lift is again on $T'$ and we denote it by $\Pmap(q)$. As a
result, we obtain a one-to-one correspondence $\Pmap : T \to T'$
which, by the analytic dependence of the leaves of $\Fol(\CC)$ on
initial conditions \cite{IY}, is a holomorphic map. In addition,
notice that $\Pmap(q_0) = q_0$ for all $a$ and $\e$.

Observe that by construction, if we consider another loop
$\delta'_0 \subset S_1$ passing thorough $q_0$ and homotopic to
$\delta_0$ on $S_1$ then the Poincar\'e map with respect to
$\delta'_0$ will be identical to $\Pmap$, possibly on a smaller
cross section $T$. This is because the homotopy between $\delta_0$
and $\delta'_0$ can be lifted to a homotopy on any leaf $\Leafaeq$
passing close enough to $S_1$. Therefore, the endpoints on $T'$ of
the lifts of $\delta_0$ and $\delta'_0$ on $\Leafaeq$ will be the
same. Similarly, $\Pmap$ does not depend on the choice of a
product structure on $N(\delta_0)$. In fact, by the tubular
neighborhood theorem \cite{Hr} any two product structures on
$N(\delta_0)$ are isotopic via an isotopy of $N(\delta_0)$ that
fixes $A(\delta_0)$ point-wise. Therefore, as a point-set, the
lift of $\delta_0$ on any near-by leaf with respect to a
projection from another product structure will be the same as the
lift obtained via $\varrho$. The difference will be only in the
parametrization of the lift.

The Poincar\'e map $\Pmap$ has the property that if two points
from the cross-section $T$ are in the same orbit of the map then
they belong to the same leaf of the foliation. Moreover, a marked
complex cycle of $\Fol(\CC),$ with a base point on $T$ and a
representative in $N(\delta_0)$ that covers $m$ times the loop
$\delta_0$ gives rise to an $m$-periodic orbit of $\Pmap$. The
converse is also true \cite{PL55}, \cite{PL57}. An $m$-periodic
orbit corresponds to a marked complex cycles of $\Fol(\CC)$ with a
base point on $T$ and a representative in $N(\delta_0)$ that
covers $\delta_0$ a number of $m$ times.

\begin{definition} \label{definition_m-fold_cycle}
A marked cycle of (\ref{ExampleFoliation}) that corresponds to an
$m$-periodic orbit of $\Pmap$ is called an $m$-fold cycle.
Whenever $m>1$ and we do not want to specify the number $m$, we
call the $m$-fold cycle a multi-fold cycle.
\end{definition}

Notice that whenever an $m-$fold cycle of $\Fol(\CC)$ corresponds
to an $m$-periodic orbit of $\Pmap$, the cycle also gives rise to
$m$ fixed points of the iterated map $\Pmap^m$.

\begin{definition} \label{definition_limit_cycle}
An $m$-fold limit cycle of $\Fol(\CC)$ is an $m$-fold cycle that
corresponds to an isolated fixed point of $\Pmap^m$.
\end{definition}

The case $m=1$ has been extensively studied. In fact, the real
cycles of a planar polynomial line field of the form
(\ref{GeneralFoliation}) extend to $1-$fold cycles of its
complexificaion. The aforementioned infinitesimal Hilbert's 16th
problem \cite{BNY}, \cite{I02} treats exactly the special case
$m=1.$ The following classical result, known as Pontryagin's
criterium \cite{Po} can be stated in the following form.

\begin{Pontryagin}Let $\delta_{\h}$ be an analytic
family of simple closed curves on the corresponding leaves
$S_{\h}$ of foliation (\ref{GeneralFoliation}) when $\e=0$.
Consider the analytic function $I(\h)= \int_{\delta_{\h}} \omega$.
If there exists a value $\h_0$ such that $I(\h_0)=0$ and $I'(\h_0)
\neq 0$ then there exists a continuous family $\delta\_{\e}$ of
loops, each representing a 1-fold complex limit cycle of
(\ref{GeneralFoliation}). Moreover, for $\e$ close to $0$, the
loops $\delta_{\e}$ always stay close to $\delta_{\h_0}$ and
$\delta_{\e} \to \delta_{\h_0}$ as $\e \to 0.$
\end{Pontryagin}

In contrast to 1-fold cycles, little is known about multi-fold
ones. That is why, the goal of this article is to shed some light
on the case $m>1$. During a series of informal discussions, Y.
Ilyashenko proposed the following questions in the spirit of
Petrovskii and Landis' works \cite{PL55} and \cite{PL57}:
\vspace{1mm}

\noindent \textbf{Q1.} {\em Are there polynomial families of form
(1) with Poncar\'e maps that have isolated periodic orbits of
arbitrary period $m>1$?} \vspace{1.5mm}

\noindent \textbf{Q2.} {\em If $m>1,$ what may happen to an
$m$-fold limit cycle when $\e$ approaches $0$?} \vspace{1.5mm}

\noindent \textbf{Q3.} {\em Does a multi-fold limit cycle settle
on a leaf of $\Folzero(\CC)$ as $\e \to 0?$} \vspace{2mm}

For the rest of this article we try to give some answers to
Ilyashenko's questions posed above for the particular family
(\ref{ExampleFoliation}). Loosely stated, the main statement of
the current paper is the following: 

\begin{main}
Multi-fold limit cycles of all possible periods appear in the
family (\ref{ExampleFoliation}) when the complex parameters $a$
and $\e$ are chosen appropriately. Moreover, each of these cycles
extends to a continuous family with respect to $\e$. Finally, when
$\e$ tends to zero, the multi-fold limit cycles from the family
escape from a very large open subdomain of $\CC$ that contains the
surface $S_1$.
\end{main}

\noindent The precise formulation of the claim above will be
stated in the next section as theorem \ref{The_Main_Theorem}.

The proof of the main result starts with a fairly explicit
construction of the Poincar\'e map $\Pmap$. After that, it is
established that periodic orbits of all periods of $\Pmap$
bifurcate from the fixed point $q_0$. This immediately yields that
$m$-fold limit cycles of all possible $m \in \Naturals$ bifurcate
from the cycle $(\Delta_0,q_0)$ of the family
(\ref{ExampleFoliation}). For the second part of the statement, we
exploit more thoroughly the connection between the topological
properties of $\Fol(\CC)$ and the dynamics of the Poincar\'e map
$\Pmap$. First, we construct a very large smooth surface
transverse to the foliation $\Fol(\CC).$ Then, we extend the
Poincar\'e map $\Pmap$ on this cross-section. We call it \emph{a
non-local Poincar\'e map}. Each multi-fold cycle of $\Fol(\CC)$
generated in the first part of our main result corresponds to a
periodic orbit of $\Pmap$ and together with that, determines a
well-defined free homotopy class of loops in an open fibred
subdomain of $\CC$. The topology and the fiber structure of this
subdomain comes from $\Folzero(\CC)$. Moreover, as the
cross-section surface is transverse to $\Fol(\CC)$, we can induce
a complex structure on it so that $\Pmap$ is holomorphic. Finally,
the construction of the non-local complex analytic Poincar\'e map
allows us to establish that the behavior of a multi-fold limit
cycle is quite different from the behavior of a 1-fold limit cycle
as $\e$ tends to zero. By Pontryagin's theorem, the latter always
stays close to some cycle from $\Folzero(\CC)$ and converges to it
as $\e$ converges to zero. In contrast to the behavior of a
$1-$fold limit cycle, a multi-fold one tends to escape from a very
large domain in $\CC$ when $\e$ approaches $0$. We call this
phenomenon a \emph{rapid evolution of the multi-fold limit cycle}.

The occurrence of quick escape of cycles is not that surprising if
one recalls the dynamics of holomorphic maps with parabolic fixed
points and their perturbations \cite{CG}, \cite{M}. In our case,
$\Pmap$ is a two parameter perturbation of the identity map. It is
well known that as $\Pmap$ approaches the identity its
$m-$periodic orbits, for $m>1$, leave the map's domain. Since a
periodic point of $\Pmap$ represents a multi-fold cycle of
$\Fol(\CC)$, the multi-fold cycles should escape too. The
challenge in our study is to establish the existence of periodic
orbits of period $m>1$ for the Poincar\'e map $\Pmap$ of the
foliation $\Fol(\CC)$ and to extend the map's domain as much as
possible. For that reason we need to analyze the link between the
foliation and the map as well as to explore the topology of
$\Folzero(\CC)$ globally which makes the problem quite more
interesting.

So far, the third question from the list above stays unanswered.
The information we have on rapid evolution reveals an interesting
insight. If the answer to that question is positive, then before a
multi-fold limit cycle can reach an algebraic leaf as $\e \to 0$,
its representatives should change their topological properties
somewhere along the way. This means that there is a possibility
that the cycle settles on a critical leaf of $\Folzero(\CC)$ or
goes through one or several critical leaves of $\Folzero(\CC),$
settling on a regular leaf. Since (\ref{GeneralFoliation}) is
polynomial, it extends to a foliation on $\cptwo.$ Thus, another
possibility is an interaction with the line at infinity.

We finish this section with a discussion about another interesting
and important issue, related to question $1$ above. A central
problem in the study of multi-fold limit cycles is their existence
in families of polynomial foliations of the form
(\ref{GeneralFoliation}). Ideally, one would like to establish
existence of multi-fold limit cycles in general families of type
(\ref{GeneralFoliation}). Heuristically, we can follow the
following steps. Using Pontryagin's theorem, we could find a
family of $1$-fold cycles which gives a family of isolated fixed
points for the corresponding Poincar\'e map $\Pmape$. For
infinitely many values of $\e$ in any neighborhood of $0$, the
derivative of $\Pmape$ evaluated at the fixed point will be an
$m$-th root of unity. Thus, for such $\e$ a local continuous
family of $m$-periodic isolated orbits could bifurcate from the
fixed point. This will happen as long as some of the resonant
terms of the map's normal form do not vanish, i.e. the map is not
analytically equivalent to a rotation. Since having all zero
resonant terms is an extremely special property for maps with
root-of-unity multiplier, we can expect that the Poincar\'e
transformations for most foliations of the form
(\ref{GeneralFoliation}) will have a lot of isolated periodic
orbits and thus, the foliations themselves will have many
multi-fold limit cycles. The only obstacle in this strategy is the
verification that some of the resonant term coefficients of the
map's normal form are nonzero. This fact imposes a challenge since
the connection between the polynomial foliation and its Poincar\'e
transformation is implicit and indirect.

\section{The main theorem} \label{Section_Main_Theorem}
Before giving a precise statement of the main result of the paper,
we are going to fix some notations and give some definitions.

Let us verify that both $S_0 = \{(x,y) \in \CC \,\, | \,\, H(x,y)
= 0\}$ and $S_1 = \{(x,y) \in \CC \,\, | \,\, H(x,y) = 1\}$ are
leaves of the foliation $\Fol$ for any value of the parameters $a$
and $\e$. Having in mind that $dH \wedge dH = 0$, consider the
wedge product
\begin{align} \label{Formula_wedge_product}
\nonumber \bigl( dH + &\e(H - 1)(y \, dx - x \, dy) + \e a \,\, y \, dH \bigr) \wedge dH = \\
\nonumber = dH \wedge dH + &\e (H-1)(y \, dx - x \, dy) \wedge dH
+ \e
a\,\,y \, dH \wedge dH =\\
= &\e (H-1)(y \, dx - x \, dy) \wedge dH = \\
\nonumber = 2 &\e (H-1)(y \, dx - x \, dy) \wedge (x \, dx + y \, dy) =\\
\nonumber = 2 &\e (H-1) H \, dx \wedge dy.
\end{align}
Since $H=0$ on $S_0$ and $H=1$ on $S_1$, the wedge product
(\ref{Formula_wedge_product}) becomes zero when restricted to
either $S_0$ or $S_1$, hence both of them are tangent to the
complex line-field $\Field$, which implies that both of them are
leaves of $\Fol(\CC)$ for any $a$ and $\e$.

Look at the polynomial $H = x^2 + y^2$ as a map $H : \CC \to \cc$.
Consider the punctured plane of regular values $B=\cc - \{0\}$ and
its preimage $E=H^{-1}(B).$ Clearly, $E$ is just $\CC$ with the
critical level set $S_0$ of $H$ removed. Recall $1 \in B$ and
hence $S_1 \subset E$, which is a topological cylinder (or a
twice punctured sphere if you prefer). 

Observe that every leaf of $\Fol(\CC)$ different from $S_0$ is
entirely contained in the domain $E$. Denote by $\Fol$ the
foliation $\Fol(\CC)$ with the leaf $S_0$ removed. Then $\Fol$ is
a foliation without singularities in $E$. In particular, when
$\e=0$ the restricted foliation $\Folzero$ consists of all fibers
of $H$ with the exception of the critical one $S_0 = H^{-1}(0)$.

Let $0< \rho_0 < R_0$, thinking of $\rho_0$ as very small and
$R_0$ as large. Define the annulus $A_0 = \{\h \in \cc \,\, | \,\,
\rho_0 < |\h| < R_0\}$. Consider the preimage $E_0 = H^{-1}(A_0)$.
Then, the set $E_0$ is the large open subdomain of $\CC$ from
which the multi-fold cycles are going to escape, according to our
main statement.

Next, we look at multi-fold cycles from a topological point of
view rather than dynamically.

\begin{definition} \label{Definition_delta_m_fold_vertical}
A loop contained in $E$ is called $m$-fold vertical provided that
it is free homotopic to $\delta_0^m$ inside the domain $E$. A
marked complex cycle of $\Fol$ is called $m$-fold vertical
provided that it has an $m$-fold vertical representative contained
in $E$.
\end{definition}

As one would expect, if a marked complex cycle of $\Fol$ has at
least one $m$-fold vertical representative in $E$, then all of its
representatives are $m$-fold vertical. Indeed, let $\delta$ and
$\delta'$ be two loops from the same marked cycle of $\Fol$ and
let $\delta$ be free-homotopic in $E$ to $\delta_0^m$. Then both
$\delta$ and $\delta'$ are homotopic on the same leaf of $\Fol$
which, in its own turn, is contained entirely in $E$. Therefore,
both representatives are homotopic to each other inside $E$ and
since $\delta$ is free-homotopic in $E$ to $\delta_0^m$, so is
$\delta'$.

More interesting is the question whether the number $m$ is a
topological invariant of an $m$-fold vertical cycle. Assume
$\delta \subset E$ is a loop representing some $m$-fold vertical
cycle of $\Fol$. Also, assume that $\delta$ is free homotopic in
$E$ to another loop $\delta'_0 \subset S_1$. As both $\delta_0^m$
and $\delta'_0$ belong to the cylinder $S_1$, whose fundamental
group is $\Integers$, the closed curve $\delta'_0$ should be free
homotopic on $S_1$ to $\delta^k$, for some $k \in \Integers$.
Therefore, the representative $\delta$ is simultaneously $m$-fold
and $k$-fold vertical. Later, in proposition
\ref{proposition_m=k}, we are going to verify that $\delta_0$ is
not null-homotopic in $E$ and $k=m$ always.

The leaves of foliation $\Fol$, given by the line field
(\ref{ExampleFoliation}), depend analytically on the two
parameters $a$ and $\e$. In order to study the phenomenon of rapid
evolution, we need to define continuous dependance of marked limit
cycles on parameters.

\begin{definition}
A family $\{(\Deltae, q_{\e})\}_{\e}$ of marked limit cycles of
$\Fol$ is called continuous with respect to $\e$ provided that
there exists a continuous family of loops $\{\deltae\}_{\e}$ such
that:

\smallskip

\noindent a) for each $\e$, the closed curve $\deltae$ belongs to
the class $\Deltae$;

\smallskip

\noindent b) the base point $q_{\e}$ varies continuously with
respect to $\e$.
\end{definition}

Let $D_r(0)=\{\e \in \cc \,\, : \,\,|\e| < r\}$ for $r>0.$ We
claim that as long as $r>0$ is chosen small enough, rapid
evolution of marked complex cycles occurs in the following form:

\begin{theorem} \label{The_Main_Theorem} For the two-parameter family
of foliations $\Fol$ given by (\ref{ExampleFoliation}) the
following statements hold:

\smallskip

\noindent 1. For any $m \in \Naturals$ large enough there exists a
complex parameter $\e_m$ near $\frac{1}{m}$ and a parameter $a_m$
such that for all $\e$ in a neighborhood of $\e_m,$ the polynomial
foliation (\ref{ExampleFoliation}) has an $m$-fold vertical limit
cycle with a representative inside the domain $E_0 \in \CC$.

\smallskip

\noindent 2. Furthermore, there is a parameter disc
$D_{r_{(m)}}(0)$ containing $\e_m$ such that for any simple curve
$\eta \subset D_{r_{(m)}}(0)$ connecting $\e_m$ to $0$ there
exists a relatively open subset $\sigma$ of $\eta,$ such that the
$m$-fold cycle from point 1 extends on $\sigma$ to a continuous
family $\{(\Delta_{\e}, q_{\e})\}_{\e \in \sigma}$ of marked
cycles of $\Fol$.

\smallskip

\noindent 3. Finally, as $\e$ moves along $\sigma$ towards $0,$ it
reaches a value $\estar \in \sigma$ such that for any $\e \in
\sigma$ past $\estar$ no $m-$fold vertical representative of
$(\Delta_{\e}, q_{\e})$ will be contained in $E_0$ anymore.
\end{theorem}

\section{The local Poincar\'e map} \label{Section_Local_Poincare_map}
We begin our investigations with the construction of the
Poincar\'e transformation locally and the computation of some of
its terms.

Define $A(\delta_0)$ as a tubular neighborhood of $\delta_0$ on
the surface $S_1$ and $N(\delta_0)$ as a tubular neighborhood of
$A(\delta_0)$ in $\CC.$ Let $$\Band_{r_0} =\{\zeta \in \cc \, : \,
|\text{Im}(\zeta)| < r_0 \}$$ be a an infinite horizontal band in
$\cc$ of width $r_0$ and let $$D_{r_0}(1) = \{\xi \in \cc \, : \,
|\xi - 1| \leq r_0\}$$ be the disc of radius $r_0$ centered at
$1$. Consider the map
$$f_1 : \Band_{r_0} \times D_{r_0}(1) \to N(\delta_0) \,\,\, \text{defined by}
\,\,\, f_1 : (\zeta,\xi) \mapsto (\xi \cos{\zeta}, \xi
\sin{\zeta}).$$ Without loss of generality, we can think that
$f_1(\Band_{r_0} \times D_{r_0}(1)) = N(\delta_0).$ In other
words, $f_1$ can be thought of as the universal covering map of
$N(\delta_0)$. Notice, that we also have $f_1(\Band_{r_0} \times
\{1\}) = A(\delta_0) \, \subset \, S_1$.

The pull-back $f_1^*\Field$ on $\Band_{r_0} \times D_{r_0}(1)$ of
the line field $\Field$ is
$$f_1^*\Field = \ker{\bigl(d(\xi^2) - \e (\xi^2 - 1)\xi^2 \, d\zeta + a \e \, \xi \sin{\zeta} \, d(\xi^2)\bigr)}.$$
For $0<r_1<1,$ define the map $$f_2 : \Band_{r_0} \times
D_{r_1}(0) \to \Band_{r_0} \times D_{r_0}(1)\,\,\, \text{where}
\,\,\, f_2 : (z,w) \mapsto \Bigl(z, \frac{1}{\sqrt{1 -
w}}\Bigr).$$ Composing the maps $f_1$ and $f_2$ we obtain $$f =
f_1 \circ f_2 \, : \, \Band_{r_0} \times D_{r_1}(0) \,
\longrightarrow \, N(\delta_0).$$ Then the pull-back $f^*\Field$
is
$$f^*\Field = \ker\left(\frac{1}{(1-w)^2}\Bigl(dw - \e\,w dz + \e a \, \frac{\sin{z}}{\sqrt{1-w}}\, dw
\Bigr)\right)$$ and since $\frac{1}{(1-w)^2}$ is well defined and
nonzero for $w \in D_{r_1}(0)$, we can cancel it out and the line
field becomes
$$f^*\Field = \ker{\Bigl(dw - \e\,w dz + \e a \, \frac{\sin{z}}{\sqrt{1-w}}\, dw
\Bigr)}.$$ The holomorphic function $\mu_{\e}(z)=e^{-\e z}$ is
nonzero everywhere, so
\begin{align*}
f^*\Field &= \ker{\Bigl(e^{-\e z} dw - \e\,w e^{-\e z}\,dz + \e
a \, \frac{e^{-\e z} \sin{z}}{\sqrt{1-w}}\, dw \Bigr)} \\
&= \ker{\Bigl(d(w e^{-\e z}) + \e a \, \frac{e^{-\e z}
\sin{z}}{\sqrt{1-w}}\, dw \Bigr)} \\
&= \ker{(d J^{(\e)} + a \omega^{(\e)})}
\end{align*}
$$\text{where} \,\,\,\, J^{(\e)} = w e^{-\e z} \,\,\,\, \text{and}\,\,\,\, \omega^{(\e)} =
\frac{e^{-\e z} \sin{z}}{\sqrt{1-w}}\, dw.$$

Our next step is to define the Poincar\'e transformation for the
foliation $\Fol$, using the local chart $f$ on the tubular
neighborhood $N(\delta_0)$ of the loop $\delta_0.$ Denote the
desired map by
$$\Pmap \, : \, D_{r_1}(0) \, \longrightarrow \, \cc.$$ We are
going to explain how it is constructed.

Define the path $\deltahat_0^{(m)} = \{(t,0) \in \Band_{r_0}
\times \{0\} \,\, | \,\, t \in [0,2\pi m]\}$ and whenever $m=1$
use the notation $\deltahat_0 = \deltahat_0^{(1)}$. Then
$f(\deltahat_0^{(m)})=\delta_0^m$. For $a$ in a neighborhood of
$0$ and for an appropriate choice of the radius $r_1$, the segment
$\deltahat_0^{(m)}$ can be lifted to a path
$\delta_{a,\e}^{(m)}(u)$ on the leaf of $\Foliation^{a,\e}$
passing through the point $(0,u) \in \{0\}\times D_{r_1}(0),$ so
that if $pr_1 : (z,w) \mapsto z$ then
$pr_1(\delta_{a,\e}^{(m)}(u))=\deltahat_0^{(m)}.$ Again, as
before, whenever $m=1$ we omit the superscript $(m)$ and we write
$\delta_{a,\e}(u) = \delta_{a,\e}^{(1)}(u)$. The lift
$\delta_{a,\e}(u)$ has two endpoints. The first one is $(0,u)$ and
the second one we denote by $(2\pi,\Pmap(u))$. When a=0, the map
$P_{0,\e}(u)$ comes from the foliation $\Foliation_{0,\e}$ which
in our tubular neighborhood is given by $\ker(d(we^{-\e z})).$
Then, $\delta_{0,\e}(u) = \{(t,u e^{\e t}) \, : \, t \in
[0,2\pi]\}$ and so $P_{0,\e}(u)=e^{2 \pi \e}u.$ Since
$\delta_{a,\e}(0)=\deltahat_0$, the equality $\Pmap(0)=0$ holds
for all $(a,\e)$. As a result, the Poincar\'e transformation can
be written down as
$$\Pmap(u)=e^{2 \pi \e}u + a I(u,\e)u + a^2 G(u,a,\e)u$$ and
its $k$-th iteration can be expressed as
$$P^k_{a,\e}(u)=e^{2 k \pi \e}u + a I_{(k)}(u,\e)u + a^2 G_{(k)}(u,a,\e)u.$$
If $\e=\frac{i}{m}$ and after $m$ iterations the map becomes
$$P^m_{a,\frac{i}{m}}(u)=u + a I_{(m)}\Bigl(u,\frac{i}{m}\Bigr)u + a^2 G_{(m)}\Bigl(u,a,\frac{i}{m}\Bigr)u.$$
In this case, denote the lift of $\deltahat_0^{(m)}$ by the
simpler notation $\delta_a^{(m)}(u) =
\delta_{a,\frac{i}{m}}^{(m)}(u)$.

In order to study the periodic orbits of $\Pmap(u)$, we are going
to look at the difference $P^m_{a,\frac{i}{m}}(u)-u$. Since
$\bigl(d J^{({i}/{m})} +
\omega^{({i}/{m})}\bigl)|_{\delta_{a}^{(m)}(u)} = 0,$ it can be
concluded that

\begin{align*}
&\int_{\delta_{a}^{(m)}(u)} \bigl(d J^{({i}/{m})} + a
\omega^{({i}/{m})}\bigr)=0 \,\,\,\, \text{and hence} \\
&\int_{\delta_{a}^{(m)}(u)} d J^{({i}/{m})} = - a
\int_{\delta_{a}^{(m)}(u)}\omega^{({i}/{m})}.
\end{align*}
The one-form $d J^{({i}/{m})}$ is exact and yields
\begin{align} \label{Formula Integral Equalities}
\nonumber P^m_{a,i/m}(u)-u &= P^m_{a,i/m}(u)e^{-2 \pi}-ue^{0}\\
\nonumber &= J^{({i}/{m})}(2 \pi m,u) - J^{({i}/{m})}(0,u) \\
          &= \int_{\delta_{a}^{(m)}(u)} d J^{({i}/{m})} \\
\nonumber &= - a \int_{\delta_{a}^{(m)}(u)}\omega^{({i}/{m})}.
\end{align} Notice that when $a=0$ the paths take the special explicit form
\begin{equation} \label{curve formula} \delta^{(m)}_{0}(u) =
\delta^{(m)}_{0,\frac{i}{m}}(u)=\{(t,e^{\frac{i}{m}t}u)\, | \, t
\in [0,2 \pi m]\} \end{equation} with endpoints $(0,u)$ and $(2
\pi m, u)$. Divide equation (\ref{Formula Integral Equalities}) by
$a$. When $a \to 0$ the limit of the left hand side of
(\ref{Formula Integral Equalities}) is $I_{(m)}(u,i/m)u$.
Moreover, $\delta_{a}^{(m)}(u) \to \delta_{0}^{(m)}(u)$ as $a \to
0$. As a result, we can conclude that
$$I_{(m)}(u,i/m)u = - \int_{\delta_{0}^{(m)}(u)}\omega^{({i}/{m})}.$$
Now, remembering that $\delta_{0}^{(m)}(u)$ is of the form
(\ref{curve formula}) compute
\begin{eqnarray*} \label{}
I_{(m)}(u,i/m)u &=& - \int_{\delta_{0}^{(m)}(u)}
\frac{e^{-\frac{i}{m}z}
\sin{z}}{\sqrt{1-w}}\, d w \\
&=& - \int_{0}^{2 \pi m} \frac{e^{-\frac{i}{m}t}
\sin{t}}{\sqrt{1-u e^{\frac{i}{m}t}}}\,\Bigl( \frac{i}{m} u e ^{\frac{i}{m}t} \Bigr) dt \\
&=& - \frac{i u}{m} \int_{0}^{2 \pi m} \frac{\sin{t}}{\sqrt{1-u
e^{\frac{i}{m}t}}} \, dt.
\end{eqnarray*}
Since both sides of the equation are divisible by $u,$
\begin{equation} \label{first approximation formula one}
I_{(m)}(u,i/m) = - \frac{i}{m} \int_{0}^{2 \pi m}
\frac{\sin{t}}{\sqrt{1-u e^{\frac{i}{m}t}}} \, dt
\end{equation}
To compute the integral in (\ref{first approximation formula
one}), notice that $1/\sqrt{1 - w}$ is well defined and
holomorphic in the disc $D_{r_1}(0) \not\ni 1$ so it expands as
uniformly convergent series
$$(1 - w)^{-\frac{1}{2}} = \sum_{k=0}^{\infty} b_k w^k ,$$ where $b_k
= (-1)^k \frac{-\frac{1}{2}\bigl(-\frac{1}{2} - 1
\bigr)\bigl(-\frac{1}{2} - 2\bigr)...\bigl(-\frac{1}{2} -
(k-1)\bigr)}{k !} \neq 0.$ Thus,
\begin{eqnarray} \label{eqn 1}
\nonumber \int_{0}^{2 \pi m} \frac{\sin{t}}{\sqrt{1-u
e^{\frac{i}{m}t}}} \, dt &=& \int_{0}^{2 \pi m} \left(
\sum_{k=0}^{\infty} b_k e^{i
\frac{k}{m} t} u^k \right) \sin{t} \, dt \\
&=& \sum_{k=0}^{\infty} b_k \left(\int_{0}^{2 \pi m} e^{i
\frac{k}{m} t} \sin{t} \, dt \right) u^k.
\end{eqnarray}
The value of the integral depends on the coefficients of (\ref{eqn
1}) which, in their own turn, depend on the integral
\begin{eqnarray*}
\int_{0}^{2 \pi m} e^{i \frac{k}{m} t} \sin{t} \, dt &=&
\frac{1}{2i} \int_{0}^{2 \pi m} e^{i \frac{k}{m} t}(e^{i t} - e^{-
i t}) \, dt \\
&=& \frac{1}{2i} \int_{0}^{2 \pi m} \bigl(e^{i\frac{k + m}{m} t} -
e^{i\frac{k - m}{m} t} \bigr) \, dt
\end{eqnarray*}
When $k \neq m$ the primitive of the function $\bigl( e^{i\frac{k
+ m}{m} t} - e^{i\frac{k - m}{m} t} \bigr)$ under the integral is
again $2 \pi m$-periodic, leading to the conclusion that the
integral is zero. When $k=m$ the integral becomes
\begin{eqnarray*}
\int_{0}^{2 \pi m} e^{i t} \sin{t} \, dt &=& \frac{1}{2i}
\int_{0}^{2 \pi m} e^{i t}(e^{i t} - e^{-
i t}) \, dt \\
&=& \frac{1}{2i} \int_{0}^{2 \pi m} \bigl(e^{i 2 t} - 1 \bigr) \,
dt \\
&=& \frac{1}{(2i)^2} \Bigl( e^{2 i t} \Bigr)_{0}^{2 \pi m} -
\frac{\pi m}{i} \\
&=& i \pi m
\end{eqnarray*}
The computations above lead to
\begin{align*}
I_{(m)}\bigl(u,\frac{i}{m}\bigr) = - \frac{i}{m} \, b_m \, i \pi m
\,\, u^m = \pi b_m \, u^m.
\end{align*}
Finally, setting $c_m = \pi b_m \neq 0,$ we can conclude that for
$\e=\frac{i}{m}$,  the Poincar\'e map takes the form
\begin{equation} \label{Equation Pmap with computed first melnikov function}
P_{a,\frac{i}{m}}^m(u) = u + a\, c_m u^{m+1} + a^2
G_{(m)}\bigl(u,a,{i}/{m}\bigr)u.
\end{equation}

\section{Existence of multi-fold cycles}
\label{Section_Proof_of_Point_1}

In this section we show how multi-fold limit cycles bifurcate from
the cycle $(\Delta_0, q_0)$, located on the leaf $S_1 =
H^{-1}(1)$. Remember that $\delta_0 \subset S_1$ is the unit
circle in the real plane $\RR \subset \CC$ centered at the origin.
The set $\Delta_0$ is the element of the fundamental group of
$S_1$ determined by the loop $\delta_0$ with a base point $q_0 =
(1,0) \in S_1$.

From the discussion in the introduction, the existence of a
multi-fold limit cycle of $\Fol$ follows from the existence of an
isolated $m$-periodic orbit of the Poincar\'e transformation
$\Pmap.$ We can see from the construction of the map that a
representative of the cycle will be contained in the tubular
neighborhood $N(\delta_0) \subset E$ and therefore free homotopic
to $\delta_0^m$ in it. This fact immediately implies that the
limit cycle will be $m$-fold vertical. Therefore, all we need to
show is that $\Pmap$ has an isolated $m$-periodic orbit.

We fix the radii $r_1 > 0, r_2>0$ and $\bar{r}_3>0$ so that for
any $(a,\e) \in D_{r_2}(0) \times D_{\bar{r}_3}(0)$ the map
$\Pmap\, : \, D_{r_1}(0) \longrightarrow \cc$ is well defined. Let
$m>0$ be such that $i/m \in D_{\bar{r}_3}(0).$

\begin{lemma} \label{Lemma Example Periodic Orbit Existence}
There exists $\e_m$ near $\frac{i}{m}$ and a parameter $a_m$ such
that for all $\e$ in a neighborhood of $\e_m,$ the map $\Pmap$ has
an isolated periodic orbit of period $m$.
\end{lemma}

\begin{proof}
The verification of the claim depends on four facts. Putting them
together will help us determine the values of the parameters $a$
and $\e.$ In order to find a periodic orbit for the map
$\Pmap(u),$ we are going to look at the equation
\begin{equation}\label{general equation}
\Pmap^m(u) - u=0.
\end{equation} Whenever $a \neq 0$ we can rewrite (\ref{general equation}) in the form
\begin{equation*}
\frac{e^{2 \pi m \e} - 1}{a} \,\, u +  I_{(m)}(u,\e) u + a \,
G_{(m)}(u,a,\e) u = 0.
\end{equation*}
Furthermore, having in mind that $u=0$ is always a solution of
(\ref{general equation}), we can divide by $u$ and obtain
\begin{equation} \label{Eq m a epsilon}
g(u,a,\e) = \frac{e^{2 \pi m \e}-1}{a} + I_{(m)}(u,\e) + a \,
G_{(m)}(u,a,\e)=0
\end{equation}
for $u \in D_{r_1}(0), a \in D_{r_2}(0) - \{0\}$ and $\e \in
D_{\bar{r}_3}(0).$ \vspace{3mm}

\paragraph{\bf Fact 1.} Let us focus on the equation
\begin{equation} \label{Eq m a i/m}
g\Bigl(u,a,\frac{i}{m}\Bigr) = I_{(m)}\Bigl(u,\frac{i}{m}\Bigr) +
a G_{(m)}\Bigl(u,a,\frac{i}{m}\Bigr)=0
\end{equation}
If necessary, decrease the radius $r_2>0$ enough so that if we set
\begin{equation*}
\mathcal{M}(r_1,r_2) = \max{\left\{|a|
\left|G_{(m)}\Bigl(u,a,\frac{i}{m}\Bigr)\right| \, \, : \, \,
|u|=r_1 \, \, \text{and} \, \, a \in D_{r_2}(0)\right\}}
\end{equation*}
then $\mathcal{M}(r_1,r_2) < |c| \, r_1^m.$ Since
$I_{(m)}\Bigl(u,\frac{i}{m}\Bigr) = c_m u^m$, it follows that for
$|u|=r_1$ and for any $a \in D_{r_2}(0)$
$$|c_m|\, |u|^m = |c_m| \, r_1^m > \mathcal{M}(r_1,r_2) \geq  |a|
\left|G_{(m)}\Bigl(u,a,\frac{i}{m}\Bigr)\right|,$$
 so by Rouche's theorem \cite{Smth}, equation (\ref{Eq m a i/m}) has
 exactly $k$ zeroes $u_1(a)$, $u_2(a),$...,$u_m(a)$ in $D_{r_1}(0)$, counted with
 multiplicities. \vspace{3mm}

\paragraph{\bf Fact 2.} Let $\mu(\e) = \min{ \{ |e^{2 \pi k \e} - 1| \,\, : \,\, 1 \leq k \leq m-1
 \}}$. Regarded as a function, $\mu(\e)$ is continuous and
 $\mu(i/m) > 0.$ Hence, there exists $r_3>0$, such that
 $\overline{D_{r_3}(i/m)} \subset D_{\bar{r}_3}(0)$. Moreover, there exists a constant $\mu > 0,$ such that
 $\mu(\e)>\mu$ for any $\e \in D_{r_3}(i/m)$. If needed,
 decrease $r_2>0$ so that
\begin{equation*} \label{Fact 2 inequality}
 \max{\bigl\{ \, |a| \, \bigl|I_{(k)}(u,\e) + a G_{(k)}(u,a,\e)\bigr| \,\, : \,\, 1 \leq k \leq m-1 \bigl\}}  <  \mu
\end{equation*}
 for all $u \in D_{r_1}(0), a \in D_{r_2}(0)$ and $\e \in
 D_{r_3}(i/m)$. \vspace{3mm}

 \paragraph{\bf Fact 3.} Equation (\ref{Eq m a epsilon}) can take
the form
\begin{eqnarray} \label{Eq m a epsilon rewritten}
g(u,a,\e) = g\Bigl(u,a,\frac{i}{m}\Bigr) + \Bigl(g(u,a,\e) -
g\Bigl(u,a,\frac{i}{m}\Bigr)\Bigr) = 0
\end{eqnarray}
For any fixed $a \in D_{r_2}(0)-\{0\}$, fact 1 reveals that
whenever $|u|=r_1,$ the following inequalities hold:
\begin{equation*}
\begin{split}
\left|g\Bigl(u,a,\frac{i}{m}\Bigr)\right| \geq
\left|I_{(m)}\Bigl(u,\frac{i}{m}\Bigr)\right| - |a| \,
\left|G_{(m)}\Bigl(u,a,\frac{i}{m}\Bigr)\right| > 0.
\end{split}
\end{equation*}
Hence, $\mu_1(a)= \min{\left\{
\left|g\Bigl(u,a,\frac{i}{m}\Bigr)\right| \, : \, |u|=r_1
\right\}}
> 0$
Notice, that for any nonzero $a \in D_{r_2}(0)$ one can find a
radius $r_3(a) > 0,$ continuously depending on $a$, such that
\begin{align*} \label{Large inequatity fact 3}
\max{\left\{\left| g(u,a,\e) - g\Bigl(u,a,\frac{i}{m}\Bigr)\right|
\, : \,  |u|=r_1, \,\,\, \e \in
D_{r_3(a)}\bigl(i/m\bigr)\right\}}< \mu_1(a),
\end{align*}
Because of the last inequality, it follows by Rouche's theorem
that equation (\ref{Eq m a epsilon}) has as many solutions as
equation (\ref{Eq m a i/m}). Thus, due to fact 1, (\ref{Eq m a
epsilon}) has exactly $m$ solutions $u_1(a,\e),...,u_m(a,\e)$,
counted with multiplicities. If we set $$W = \bigsqcup\limits_{0
\neq a \in D_{r_2}(0)} \,\Bigl( \{a\} \times
D_{r_3(a)}(i/m)\Bigr),$$ then $W$ is open and $\overline{W} \ni
(0,\frac{i}{m}).$\vspace{3mm}

\paragraph{\bf Fact 4.} Let $g_0(a,\e) = (e^{2 \pi m \e} - 1) + a
\, I_{(m)}(0,\e) + a^2 \, G_{(m)}(0,a,\e)$. Notice, that
$g_0\bigl(0, \frac{i}{m}\bigr) = 0 \,\,\, \text{and} \,\,\,
\frac{\partial g_0}{\partial \e}\bigl(0, \frac{i}{m}\bigr) = 2 \pi
m \neq 0.$ Hence, by the implicit function theorem, it follows
that for a possibly decreased $r_2>0$ there exists a holomorphic
function $\chi : D_{r_2}(0) \to D_{r_3}\bigl(\frac{i}{m}\bigr)$
such that $\chi(0) = \frac{i}{m}$ and $g_0(a,\chi(a))=0$ for all
$a \in D_{r_2}(0).$ From here, we can see that the zero locus of
$g_0$ inside the product domain $D_{r_2}(0) \times
D_{r_3}\bigl(\frac{i}{m}\bigr)$ is
$$Z = \{(a,\e) \,\, : \,\, g_0(a,\e)=0\} =
\{(a,\chi(a)) \,\, : \,\, a \in D_{r_2}(0)\}.$$ The set $Z$ is
relatively closed in $D_{r_2}(0) \times
D_{r_3}\bigl(\frac{i}{m}\bigr)$ so its complement
$\bigl(D_{r_2}(0) \times D_{r_3}\bigl(\frac{i}{m}\bigr)\bigl) - Z$
is open and nonempty. Therefore, $W \cap \Bigl[\bigl(D_{r_2}(0)
\times D_{r_3}\bigl(\frac{i}{m}\bigr)\bigl) - Z \Bigr] \neq
\varnothing$ is open as well. \vspace{3mm}

Now we are ready to complete the proof of the lemma. Let
$(a_m,\e_m) \in W \cap \Bigl[\bigl(D_{r_2}(0) \times
D_{r_3}\bigl(\frac{i}{m}\bigr)\bigl) - Z \Bigr]$. Apply the
results from fact 4 to obtain
\begin{align*}
  g_0(a_m,\e_m) = (e^{2 \pi m \e_m} - 1) &+ a_m \, I_{(m)}(0,\e_m) +\\
                                         &+ a^2_m \, G_{(m)}(0,a_m,\e_m) \neq 0.
\end{align*}
Hence, the equation
\begin{align*}
  P_{\e_m,a_m}^m(u) - u = (e^{2\pi m \e_m} - 1) \, u
  &+ a_m \,I_{(m)}(u,\e_m)u +\\
  &+ a^2_m \, G_{(m)}(u,a_m,\e_m)u = 0
\end{align*}
has $u_0=0$ as a simple root.

Since $(a_m,\e_m) \in W$, it follows from fact 3 that whenever
$|u|=r_1$ the following inequality holds
\begin{align*}
\left|g\Bigl(u,a_m,\frac{i}{m}\Bigr)\right| \geq \mu_1(a_m)
> \left| g(u,a_m,\e_m) - g\Bigl(u,a_m,\frac{i}{m}\Bigr)\right|
\end{align*}
Therefore, by Rouche's theorem, the equation
\begin{equation} \label{equation a_m eps}
  \begin{split}
   a_m \, g(u,a_m,\e_m) = (e^{2 \pi m \e_m} - 1)
   &+ a_m \, I_{(m)}(u,\e_m) + \\
   &+ a^2_m \, G_{(m)}(u,a_m,\e_m) = 0
   \end{split}
\end{equation}
has as many solutions as
\begin{equation} \label{equation a_m im}
a_m \, g\Bigl(u,a_m,\frac{i}{m}\Bigr) = a_m \,
I_{(m)}\Bigl(u,\frac{i}{m}\Bigr) + a_m^2 \,
G_{(m)}\Bigl(u,a_m,\frac{i}{m}\Bigr) = 0.
\end{equation}
By fact 1, equation (\ref{equation a_m im}) has $m$ roots
$u_1(a_m),...,u_m(a_m)$ contained in $D_{r_1}(0)$. For that
reason, equation (\ref{equation a_m eps}) has $m$ solutions
contained in $D_{r_1}(0).$ Let us denote them by $u_1(a_m,\e_m),$
...,$u_m(a_m,\e_m).$ As it was established earlier, none of them
is zero. For simplicity, let $u_j = u_j(a_m,\e_m),$ where $j =
1,..,m.$

By fact 2, for $1 \leq k \leq m-1$ and for $u \in D_{r_1}(0)$,
$$|e^{2 \pi k \e_m} - 1| \geq \mu(a_m) > \mu > |a_m| \, \bigl|I_{(k)}(u,\e_m) + a_m G_{(k)}(u,a_m,\e_m)\bigr|.$$
Having in mind that $u_j \in D_{r_0}(0)$ and each of them is
nonzero for $j=1,..,m$, we estimate
\begin{align*}
\bigl|P_{a_m,\e_m}^k(u_j)-u_j\bigr| =& |u_j| \, \bigl|(e^{2 \pi k
\e_m} - 1) + a_m \, I_{(k)}(u_j,\e_m) \\
&+ a^2_m \, G_{(k)}(u_j,a_m,\e_m)\bigr| \geq |u_j| \, \bigl( |e^{2
\pi k \e_m} - 1| \\ &- |a_m| \, \bigl|I_{(k)}(u_j,\e_m) + a^2_m \,
G_{(k)}(u_j,a_m,\e_m)\bigr|\bigr) > 0.
\end{align*}
For that reason, $P^k_{a_m,\e_m}(u_j) \neq u_j$ for $1 \leq k \leq
m-1$. Hence, the orbit $u_1$,...,$u_m$ consists of different
points and therefore is periodic of period $m$ in $D_{r_1}(0).$
\end{proof}

\section{Topology of the fiber bundle} \label{Section_Topology_of_Bundle}

Some of the constructions we would need in order to complete the
proof of theorem \ref{The_Main_Theorem} depend on the topology of
the domain $E \subset \CC$. That is why our goal is to understand
it well.

We begin with the introduction of some useful notations. By
$[\delta]_M$ we denote the set of all loops homotopic to a loop
$\delta$ on a manifold $M$, all passing through a base point $x_0
\in M$. As usual, the homotopy class $[\delta]_M$ is an element of
the fundamental group $\pi_1(M,x_0)$ of $M$. Our first step is to
compute the fundamental group of the domain $E$.

To make our arguments more standard, we introduce new coordinates
in $\CC$. Let $z = \frac{1}{\sqrt{2}}(x + iy)$ and $w =
\frac{1}{\sqrt{2}}(x - iy)$. This is a unitary linear
transformation of $\CC$ and it preserves the standard Hermitian
dot product $z_1\overline{z}_2 + w_1 \overline{w}_2$. Therefore,
the change of variables is an isometry and preserves all the
metric properties of $\CC$. With respect to these coordinates $H =
x^2 + y^2 = zw$ and $E = \{(z,w) \in \CC\ \,\, | \,\, z w \neq
0\}$. Also, remember the unit circle $\delta_0 = \{(\cos(2 \pi t),
\sin(2 \pi t)) \,\, | \,\, t \in [0, 1]\}$ in $x,y$-coordinates.
In $z,w$-coordinates, it takes the form $\delta_0 =
\{(\frac{1}{\sqrt{2}} e^{2 \pi t}, \frac{1}{\sqrt{2}} e^{- 2 \pi
t}) \,\, | \,\, t \in [0, 1]\}$.



\begin{lemma} \label{Lemma_defomrmation_retract_onto_torus}
The open domain $E = \{(z,w) \in \CC\ \,\, | \,\, z \neq 0 \,\,
\text{and} \,\, w \neq 0\}$ deformation retracts onto the embedded
torus $\Torus = \{(z,w) \in \CC \,\, | \,\,\, |z| = |w| =
\frac{1}{\sqrt{2}} \}$. Moreover, the circle $\delta_0$ lies on
$\Torus$ and is not null-homotopic on it.
\end{lemma}

\begin{proof}

Let $\ThreeSphere = \{(z,w) \in \CC \,\, | \,\,\, |z|^2 + |w|^2 =
1\}$ be the unit three-sphere in $\CC$ and $K = \ThreeSphere \cap
\{(z,w) \in \CC  \,\, | \,\, z=w=0\}$. Then $\Torus$ is embedded
in $\ThreeSphere$. The intersection of a complex line with
$\ThreeSphere$ is always a great circle, and more precisely, a
fiber of the Hopf bundle \cite{T}. Therefore, $K$ is the classical
Hopf link in the three-sphere, consisting of two great circles
linked once. Let $M_K = \ThreeSphere \setminus K$ be its
complement. As $\CC \setminus \{(0,0)\}$ deformation retracts onto
$\ThreeSphere$, we have that $E$ deformation retracts onto $M_K$.
In its own turn, $M_K$ deformation retracts onto the torus
$\Torus$, so we conclude that $E$ deformation retracts onto
$\Torus$.

For the second part of the lemma, notice that $\Torus =
\{(\frac{1}{\sqrt{2}} e^{2 \pi s_1}, \frac{1}{\sqrt{2}} e^{2 \pi
s_2}) \, |\, (s_1,s_2) \in [0, 1]^2\}$. From here, immediately
follows that $\delta_0$ lies on the surface of the torus. If we
take an $s_1$-circle and an $s_2$-circle on $\Torus$ as a homology
basis for $H_1(\Torus,\Integers)$ then $\delta_0$ has homology
coordinates $(1,-1)$ and therefore is not null-homologous. Hence,
it is not null-homotopic either.
\end{proof}

\begin{corollary} \label{Corollary_fundamental_group_E}
The fundamental group of $E $ is $\pi_1(E,q_0) = \Integers \oplus
\Integers$ and $[\delta_0]_E \neq 1$.
\end{corollary}

\begin{proof}
By lemma \ref{Lemma_defomrmation_retract_onto_torus}, the domain
$E$ deformation retracts onto the embedded torus $\Torus$ which
induces an isomorphism between the fundamental groups
$\pi_1(E,q_0)$ and $\pi_1(\Torus,q_0) = \Integers \oplus
\Integers$ \cite{H}. The circle $\delta_0$ is kept point-wise
fixed by the deformation retraction so $[\delta_0]_E$ gets mapped
to $[\delta_0]_{\Torus} \neq 1$. Hence $[\delta_0]_E \neq 1$.
\end{proof}

\begin{proposition} \label{proposition_m=k}
Let $\delta \subset E$ be an $m$-fold vertical representative of
an $m$-fold vertical cycle of $\Fol$. Assume that $\delta$ is also
free homotopic in $E$ to another loop $\delta'_0 \subset S_1$.
Then $\delta'_0$ is free homotopic on $S_1$ to $\delta_0^m$.
\end{proposition}

\begin{proof}
Both $\delta_0^m$ and $\delta'_0$ belong to the cylinder $S_1$,
whose fundamental group is $\Integers$. That is why, the closed
curve $\delta'_0$ should be free homotopic on $S_1$ to
$\delta_0^k$, for some $k \in \Integers$. All we need to show is
that $m=k$. Indeed, consider the corresponding elements
$[\delta_0^m]_E = [\delta_0]_E^m$ and $[\delta_0^k]_E =
[\delta_0]_E^k$ form the fundamental group $\pi_1(E,q_0)$. Both
loops $\delta_0^k$ and $\delta_0^m$ are free-homotopic in $E$ to
the same loop $\delta$. Hence, they are free-homotopic to each
other in $E$. Therefore, the elements $[\delta_0]_E^m$ and
$[\delta_0]_E^k$ are conjugate in $\pi_1(E,q_0)$. Since, according
to corollary \ref{Corollary_fundamental_group_E}, the fundamental
group $\pi_1(E,q_0) = \Integers \oplus \Integers$ is Abelian,
$[\delta_0]_E^m = [\delta_0]_E^k$ which is equivalent to
$[\delta_0]_E^{m - k} = 1$ (here we use multiplicative notation).
By the same corollary, we can see that $[\delta_0]_E \neq 1$ and
$\pi_1(E,q_0) = \Integers \oplus \Integers$ is torsion-free.
Therefore $m=k$
\end{proof}

The map $H : E \to B$ defined by the polynomial $H=x^2+y^2$ is a
smooth, locally trivial fiber bundle. For any regular value $\h
\in B$, the fibers $S_{\h} = \{p \in \CC \,\, | \,\, H(p) = \h \}$
are topological cylinders, diffeomorphic to each other \cite{AGV},
\cite{IY}. We use $S_1$ as a model fiber. The map $\nu : \cc \to
S_1$ given by $\nu(\zeta) = (\cos \zeta, \sin \zeta)$, for $\zeta
\in \cc$, is the universal covering map of $S_1$. Define
$\hat{C}_0 = \{ \zeta \in \cc \,\, | \,\, \frac{1}{2} <
\textrm{Im}(\zeta) < \frac{3}{2}\}$. Then $C_0 = \nu(\hat{c}_0)
\subset S_1$ is a non-trivial cylinder on $S_1$ such that
$\delta_0 \cap \overline{C}_0 = \varnothing$.

Here are some facts about the topology of the fiber bundle $H : E
\to B.$ The unit circle $\gamma_0 = \{\h \in \cc \,\,|\,\,
|\h|=1\}$ is a simple closed loop in $B$ starting from $1$, going
around $0$ counterclockwise and coming back to $1$. The homotopy
class of $\gamma_0$ with a base point $1$ is the generator of the
fundamental group $\pi_1(B, 1) \cong \Integers$. For $\h \in
\gamma_0$ consider the fiber $S_{\h}$. Then, if the parameter $\h$
starts from $1$ and moves along the loop $\gamma_0$ until it comes
back to $1$ then the corresponding fiber $S_{\h}$ will also make
one turn around the critical value $0$ starting and ending up at
$S_1$. This procedure gives rise to an isotopy class of
diffeomorphisms with a representative $\tilde{D}_0 : S_1 \to S_1$
which is a Dehn twist. The map $\tilde{D}_0$ can be chosen so that
it twists the cylinder $C_0$ and is the identity on $S_1 \setminus
\overline{C}_0$ \cite{AGV}.

Let $\pi(z) = e^{2 \pi i z}$. Then $\pi : \cc \to B$ is the
universal covering map of the punctured plane $B$. Denote its
group of deck transformations by $\Gamma = \{\gamma_0^m \,\,|\,\,
m \in \Integers \,\, \text{and} \,\, \gamma_0(z) = z + 1\}$. As
usual, it is isomorphic to the fundamental group of $B$. The
vertical strip $\Band = \{z \in \cc \,\, | \,\, 0 \leq
\text{Re}(z) \leq 1\}$ is a closed fundamental domain. 
Also, whenever we have a cartesian product $M_1 \times M_2$ of two
sets, by $pr_{M_j}$ we are going to denote the projection
$pr_{M_j} : M_1 \times M_2 \to M_i$ where $pr_{M_j}(m_1,m_2)=m_j$
for $j=1,2.$

The next lemma shows that we can "unfold" the bundle $H : E \to B$
into a trivial covering bundle $pr_{\cc} : \cc \times S_1 \to
\cc$. Moreover, we can do so by making sure the deck group
$\Gammahat$ acts in a very special manner. It not only takes
vertical fibers $\{z\} \times S_1$ to vertical fibers, but what
really important is that it also takes horizontal fibers $\cc
\times \{p\}$ to horizontal fibers.

\begin{lemma} \label{lemma_bundle_topology}
There is a smooth covering map $\Pi : \CtimesS \to E$ with the
following properties: \vspace{1.5mm}

\noindent {\em 1.} If $pr_{\cc} : \cc \times S_1 \to \cc$ is the
projection $(z,p) \mapsto z$ then $H \circ \Pi =\pi \circ
pr_{\cc}$. In other words, apart from being a covering map, $\Pi$
is also a bundle map. \vspace{1.5mm}

\noindent {\em 2.} The deck group of $\Pi : \cc \times S_1 \to E$
is
$$\Gammahat = \langle \,\, (z,p) \mapsto (\gamma_0(z),
D_0(p)) \,\, \rangle,$$ where $\gamma_0(z) = z+1$ is the earlier
described generator of $\Gamma$ and the map $D_0=\tilde{D}_0^{-1}$
is the Dehn twist of the cylinder $c_0 \subset S_1$ and identity
everywhere else on the surface $S_1.$ Thus, the factor bundle
$(\cc \times S_1) / \hat{\Gamma}$ is
diffeomorphically isomorphic to the bundle $E.$ 
\end{lemma}

\begin{proof}
Consider the pullback of the bundle $H : E \to B$ over the plane
$\cc$ under the covering map $\pi$. To carry out this
construction, first define the total space $\pi^{*}E=\{(z,q) \in
\cc \times E \,\, | \,\, \pi(z)=H(q)\}$. Then, the restricted
projection $\kappa = (pr_{\cc})_{|_{\pi^{*}E}} : \pi^{*}E \to \cc$
gives us the desired pullback bundle. Also, there is a map
$\tilde{\Pi}^{'}=(pr_{E})_{|_{\pi^{*}E}} : \pi^{*}E \to E$ that
satisfies the condition $H \circ \tilde{\Pi}^{'} = \kappa \circ
\pi$ and so it is a bundle map over the map $\pi.$ Together with
that, $\tilde{\Pi}^{'} : \pi^{*}E \to E$ is a covering map.

Because $\cc$ is contractible, the pullback bundle $\kappa :
\pi^{*}E \to \cc$ is trivializible, i.e. there is a smooth bundle
isomorphism $\varsigma : \cc \times S_1 \to \pi^{*}E$ so that
$\kappa \circ \varsigma = pr_{\cc} \circ id_{\cc}$ where
$id_{\cc}$ is the identity map on $\cc$. Then, the composition
$\tilde{\Pi}=\tilde{\Pi}^{'} \circ \varsigma : \CtimesS \to E$
satisfies the condition $H \circ \tilde{\Pi}= \pi \circ pr_{\cc}$
and therefore is a bundle map and a covering map at the same time.
Without loss of generality we can think that $\tilde{\Pi}(0,p)=p,$
that is we identify the fiber $\{0\} \times S_1$ with the surface
$S_1$, where $\pi(0)=1$.

We are going to look at the deck group $\Gammatilde$ of the
covering map $\tilde{\Pi}.$ For any $\gammatilde \in \Gammatilde$
we have the relation $pr_{\cc} \circ \gammatilde = \gamma_0^m
\circ pr_{\cc}$ for some $m \in \Integers$. That is why, just like
$\Gamma$, the group $\Gammatilde$ is free Abelian with one
generator $\gammatilde_0(z,p) = (\gamma_0(z), \psi(z,p))$, where
$(z,p) \in \CtimesS$. The map $\psi : \CtimesS \to S_1$ is smooth
and if we use the notation $\psi_z(p) = \psi(z,p)$, then for any
fixed $z\in \cc$ the resulting map $\psi_z : S_1 \to S_1$ is a
diffeomorphism on the standard fiber $S_1$. If we factor
$\CtimesS$ by the action of the deck group $\Gammatilde$ we obtain
the manifold $(\CtimesS) / \Gammatilde$ which is isomorphic to $E$
as a fiber bundle over $B$.

Consider a thin open strip $N_{\epsilon} = \{z \in \cc \,\, |
\,\,\,|\text{Re}(z)| < \epsilon \}$ where $\epsilon <
\frac{1}{3}$. Let $N'_{\epsilon} = \gamma_0(N_{\epsilon})$ and
take $\tilde \Band = \Band \cup N_{\epsilon} \cup N'_{\epsilon}$.
Then we can regard the smooth map $\phi_0 : N_{\epsilon} \times
S_1 \to N'_{\epsilon} \times S_1$, defined by the expression
$\phi_0(z,p) = \gammatilde_0(z,p) =(\gamma_0(z),\psi(z,p))$ for
any $(z,p) \in N_{\epsilon}\times S_1$, as a gluing map. In other
words, since $\gammatilde_0$ respects the bundle structure of
$\CtimesS$, the quotients $(\tilde{\Band} \times S_1)/\phi_0$ and
$(\CtimesS)/\Gammatilde$ are smoothly isomorphic as fiber bundles
over $B$ (for isotopies of gluing maps, see for example
\cite{Hr}). Therefore, $(\tilde{\Band} \times S_1)/\phi_0$ and $E$
are smoothly isomorphic as bundles over $B$.

The strip $N_{\epsilon}$ deformation retracts onto the point $0
\in N_{\epsilon}$. Therefore, there exists a smooth deformation
retraction $\vartheta : N_{\epsilon} \times [0, \frac{1}{2}] \to
N_{\epsilon}$. Then, for $t=0$ the map $\vartheta_t : N_{\epsilon}
\to N_{\epsilon}$ is the identity on $N_{\epsilon}$, for $t=1$ it
is the constant map $\vartheta_1 \equiv 0$ and for all $t \in [0,
\frac{1}{2}]$ we have that $\vartheta_t(0)=0$. With the help of
$\vartheta_t$, we define the isotopy
\begin{align*}
 &\phi : N_{\epsilon} \times S_1 \times [0, 0.5] \to
 N'_{\epsilon} \times S \\
 &\phi_t(z,p)=\bigl(\gamma_0(z),\psi(\vartheta_t(z),p)\bigr).
\end{align*}
As a result, when $t=0$ we have the earlier defined map $\phi_0.$
Moreover, when $t=0.5$ we obtain the map
$\phi_{0.5}(z,p)=(\gamma_0(z),\psi(0,p))$ for $(z,p) \in
N_{\epsilon} \times S_1$. Notice that the second component of
$\phi_{0.5}$ does not depend on the variable $z$ but only on $p.$
Whenever $z=0$, the map $\psi_0(p) = \psi(0,p)$ is isotopic to the
Dehn twist $D_0=\tilde{D}_0^{-1}$. This follows from
Picard-Lefchetz's theory as discussed previously in the current
section and in \cite{AGV}. Thus, we can extend smoothly the
isotopy $\phi_t$ for $t \in [0,1]$ so that for $t=0$ the gluing
map is $\phi_0$, for $t=0.5$ the map becomes $\phi_{0.5}$ from
above and finally when $t=1$ we obtain $\phi_1(z,p) =
(\gamma_0(z), D_0(p))$ for all $(z,p) \in N_{\epsilon} \times
S_1$.

Notice that $\phi_t$ respects the vertical fibers $\{z\} \times
S_1,$ that is the isotopy takes place only with respect to the
second coordinate, along the surface $S_1,$ while the first
coordinate is kept the same. Therefore, $(\tilde{\Band} \times
S_1)/\phi_0$ and $(\tilde{\Band} \times S_1)/\phi_1$ are smoothly
isomorphic as fiber bundles over $B.$ As we already saw,
$(\tilde{\Band} \times S_1)/\phi_0$ and $E$ are isomorphic as
well. Hence, $(\tilde{\Band} \times S_1)/\phi_1$ and $E$ are
isomorphic as bundles over $B.$ Since by construction
$(\tilde{\Band} \times S_1)/\phi_1$ and $(\cc \times
S_1)/\Gammahat$ are also isomorphic as bundles over $B$, we can
conclude that there exists a smooth bundle isomorphism $\Phi :
(\CtimesS)/\Gammahat \to E.$ If $\upsilon : \CtimesS \to
(\CtimesS)/\Gammahat$ is the quotient map, then it is a bundle map
over the covering map $\pi.$ When we compose it with $\Phi$ we
obtain the desired bundle covering map $\Pi = \Phi \circ \upsilon
: \CtimesS \to E$ satisfying the condition $H \circ \Pi = \pi
\circ pr_{\cc}$ and having $\Gammahat$ as its group of deck
transformations. This completes the proof of the lemma.
\end{proof}

The results from lemma \ref{lemma_bundle_topology} are a main tool
in the proof of theorem \ref{The_Main_Theorem}. As it was
mentioned already, a deck transformation $\hat{\gamma_0}^m(z,p) =
(\gamma_0^m(z), D^m_0(p))$ from $\hat{\Gamma}$ maps not only
vertical fibers $\{z\} \times S$ to vertical fibers
$\{\gamma^m_0(z)\} \times S_1$ but also horizontal fibers $\cc
\times \{p\}$ to horizontal fibers $\cc \times \{D^m_0(p)\}$. In
particular, since $D^m_0$ acts on $S_1 - c_0$ as the identity map,
whenever $p \in S_1 \setminus C_0,$ the horizontal plane $\cc
\times \{p\}$ is invariant under the action of $\Gammahat.$ These
facts lead us to the following conclusion.

\begin{corollary} \label{transverse surface}
For $p \in S_1 \setminus C_0$, the projection $\Pi(\cc \times
\{p\}) = B_p$ is a smoothly embedded surface in $E,$ diffeomorphic
to the punctured plane $B.$ Moreover, $B_p$ intersects each leaf
from the integrable foliation $\Folzero$ transversely at a single
point.
\end{corollary}

In particular, this corollary applies to the point $q_0=(1,0)$.
Thus, we have obtained a global cross-section $B_{q_0}$.

\section{The non-local Poincar\'e map}
\label{Section_{Map_Periodic_Orbits_Cycles}}

For the rest of the article, we are going fix some $m \in
\Naturals$ and take the parameter $a=a_m$ as in lemma \ref{Lemma
Example Periodic Orbit Existence}. Since $a=a_m$ is fixed, from
now on we are going to drop $a$ from all notations that contain it
either as a subscript or a superscript. For example, when $a=a_m$
we will write $\Fole$ instead of $\Fol$, the folaition leaves will
be denoted by $\Leafe$ instead of $\Leafae$ and the notation
$\Pmape$ will replace the previously accepted notation $\Pmap$ for
the Poincar\'e map.

Apart from the annulus $A_0$ from section
\ref{Section_Main_Theorem}, define two more domains in $B$.
Remember that for a pair of numbers $0< \rho < R$ we defined the
annulus $A(\rho,R) = \{\h \in \cc \,\, | \,\, \rho < |\h| < R\}$.
Fix the small positive numbers $\rho_0
>\rho'_0 > \rho'_1 > 0$ and the large ones $0 < R_0 <
R'_0 < R'_1$. Recall that $A_0 = A(\rho_0, R_0)$. Denote by $A'_0$
and $A'_1$ the annuli $A(\rho'_0,R'_0)$ and $A(\rho'_1,R'_1)$
respectively. As a result we obtain three nested open sets $A_0
\subset A'_0 \subset A'_1$.
$$\text{Recall that} \,\, E_0 = H^{-1}(A_0) \,\, \text{and define} \,\,  E'_1 = H^{-1}(A'_1).$$
Next, we lift on $\cc$ all annuli form the previous paragraph to
obtain the three horizontal strips $\pi^{-1}(A_0), \,
\pi^{-1}(A'_0)$ and $\pi^{-1}(A'_1)$. We fix the following
cross-sections in $\CtimesS$: $$\Ahatzero = \pi^{-1}(A_0) \times
\{q_0\}, \,\,\, \Ahatzeroprime = \pi^{-1}(A'_0) \times \{q_0\}
\,\,\, \text{and} \, \,\, \Ahatoneprime = \pi^{-1}(A'_1) \times
\{q_0\}.$$ All of them are subsets of $\cc \times \{q_0\}$. Let
$$\Pi_0 = \Pi|_{\cc \times \{q_0\}} \,\, : \,\, \cc \times \{q_0\}
\, \longrightarrow \, B_{q_0}.$$ Projecting by $\Pi_0$, let
$$A_0(q_0) = \Pi_0(\Ahatzero), \,\, A'_0(q_0) = \Pi_0(\Ahatzeroprime) \,\,\,
\text{and} \,\,\, A'_1(q_0)=\Pi_0(\Ahatoneprime),$$ all of which
are subsets of the embedded in $E$ surface $B_{q_0}$.

From now on, we also use the shorter notations $\hat{z} = (z, q_0)
\in \cc \times \{q_0\}$.

Consider the pulled-back foliation $\Folehat=\Pi^*\Fole$ on the
covering space $\CtimesS$. It is invariant with respect to the
action of $\Gammahat$. In other words, if $\gammahat \in
\Gammahat$ and $\Leafhatezp$ is a leaf of $\Folehat$ passing
through the point $(z,p) \in \CtimesS,$ then $\gammahat
(\Leafhatezp) = \Leafhate(\gammahat(z,p))$.

Notice that the closure of the projection
$\Pi(\Ahat'_1)=A'_1(q_0)$ is compact in $E$ and the tangent field
$\Fielde$ is transverse to the embedded cylinder $A'_1(q_0)$ for
all $|\e| \leq r$, where $r>0$ is chosen small enough. For the
next lemma we need the strip $Q = \gamma_0^{-1}(\Band) \cup \Band
\cup \gamma_0(\Band)$ consisting of three adjacent copies of the
closed fundamental domain of $\Gamma$. Take $Q'_0 = \overline{Q
\cap \pi^{-1}(A'_0)}$ and let $\hat{Q}'_0 = Q'_0 \times \{q_0\}$.

\begin{lemma} \label{lemma about lifted Pmap}
For a small enough $r>0$ and for any $|\e| \leq r$ there exists a
smooth Poincar\'e map $\Pmaphate : \hat{Q}'_0 \to \Ahatoneprime$
associated with the foliation $\Folehat$ such that for any
$\gammahat \in \Gammahat$ if both points $\hat{z}$ and
$\gammahat(\hat{z})$ belong to $\hat{Q}'_0$ then $\gammahat \circ
\Pmaphate = \Pmaphate \circ \gammahat.$
\end{lemma}

\begin{proof} As usual, let $pr_{S_1} : \CtimesS \to S_1$ be the
projection $(z,p) \mapsto p.$ By continuous dependance of
$\Folehat$ on parameters and initial conditions, we can choose the
radius $r$ of the parameter space so that the construction that
follows holds for any $|\e| \leq r$. Choose an arbitrary point
$\hat{z} \in \hat{Q}'_0.$ If $\Leafhate(\hat{z})$ is the leaf of
the perturbed foliation $\Folehat,$ passing through $\hat{z} =
(z,q_0),$ lift the loop $\delta_0$ to a curve
$\hat{\delta}_{\e}(\hat{z})$ on $\Leafhate(\hat{z})$ so that
$\hat{\delta}_{\e}(\hat{z})$ covers $\delta_0$ under the
projection $pr_{S_1}.$ Since $r$ is chosen small enough, the lift
$\hat{\delta}_{\e}(\hat{z})$ is contained in the domain
$\pi^{-1}(A'_1) \times S_1$ and both of its endpoints are on
$\Ahatoneprime.$ The first endpoint is $\hat{z} \in \hat{Q}'_0$
and the second we denote by $\Pmaphate(\hat{z}) \in
\Ahatoneprime.$ Thus, we obtain the correspondence $\Pmaphate :
\hat{Q}'_0 \to \Ahatoneprime$, which is a smooth map close to
identity.

By construction, the cross-section $\Ahatoneprime$ is
$\Gammahat-$invariant. Now, let $\hat{z} \in \hat{Q}'_0$ and
assume that $\gammahat(\hat{z}) \in \hat{Q}'_0$ for some
$\gammahat \in \Gammahat.$ As pointed out earlier, the arc
$\hat{\delta}_{\e}(\hat{z})$ is the lift of $\delta_0$ on
$\Leafhate(\hat{z})$ under the projection $pr_{S_1}.$ It connects
the two points $\hat{z} \in \hat{Q}$ and $\Pmaphat(\hat{z}) \in
\Ahatoneprime.$ The image $\gammahat(\hat{\delta}_{\e}(\hat{z}))$
lies on the leaf $\gammahat(\Leafhate(\hat{z})) =
\Leafhate(\gammahat(\hat{z}))$ and its endpoints are
$\gammahat(\hat{z}) \in \hat{Q}'_0$ and
$\gammahat(\Pmaphate(\hat{z})) \in \Ahatoneprime.$ Since
$\Gammahat \cong \Integers$, its element $\gammahat =
\gammahat_0^k$ for some $k \in \Integers$. Therefore $pr_{S_1}
\circ \gammahat(z,p)=pr_{S_1}(\gamma_0^k(z),D_0^k(p))=D_0^k(p)=
D_0^k \circ pr_{S_1}(z,p)$ for any $(z,p) \in \CtimesS$. The fact
that $\hat{\delta}_{\e}(\hat{z})$ is the lift of $\delta_0$ on the
leaf $\Leafhate(\hat{z})$ from $\Folehat$ means that
$pr_{S_1}(\hat{\delta}_{\e}(\hat{z}))=\delta_0.$ Similarly, to
find out what the arc $\gammahat(\hat{\delta}_{\e}(\hat{z}))$ is a
lift of, we just have to project it onto $S_1.$ Using the property
$pr_{S_1} \circ \gammahat_0^k = D_0^k \circ pr_{S_1}$ we conclude
that $pr_{S_1} \circ
\gammahat_0^k(\hat{\delta}_{\e}(\hat{z}))=D_0^k \circ
pr_{S_1}(\hat{\delta}_{\e}(\hat{z}))=D_0^k(\delta_0).$ Since $D_0$
acts like the identity everywhere on $S_1$ except for the thin
cylinder $C_0 \subset S_1$ and $\delta_0 \cap C_0 = \varnothing$,
it immediately follows that $D_0^k(\delta_0) = \delta_0$.
Therefore $\gammahat(\hat{\delta}_{\e}(\hat{z}))$ is the lift of
$\delta_0$ on the leaf $\Leafhate(\gamma(\hat{z}))$ under the
projection $pr_{S_1}.$ Whit this in mind, the endpoint
$\gammahat(\Pmaphate(\hat{z}))$ can also be rewritten as
$\Pmaphate(\gammahat(\hat{z})).$ Thus, we obtain the relation
$\gammahat \circ \Pmaphate = \Pmaphate \circ \gammahat.$
\end{proof}

Lemma \ref{lemma about lifted Pmap} allows us to extend
$\Pmaphate$ from a map on $\hat{Q}'_0$ to a $\Gammahat-$equivarint
map on the whole cross-section $\Ahatzeroprime \subset \cc \times
\{q_0\}.$ This fact makes it possible for the $\Pmaphate$ to
descend under the covering $\Pi_0 : \Ahatzeroprime \to
\Azeroprimeq$ to a Poincar\'e map defined on $\Azeroprimeq \subset
E.$

\begin{corollary} \label{Corollary Pmaphat extension}
The transformation $\Pmaphate$ constructed in lemma \ref{lemma
about lifted Pmap} extends to a map $\Pmaphate : \Ahatzeroprime
\to \Ahatoneprime$ for the foliation $\Folehat$ such that for any
$\gammahat \in \Gammahat$ the equivariance relation $ \gammahat
\circ \Pmaphate = \Pmaphate \circ \gammahat$ holds.
\end{corollary}

\begin{proof}
By construction, both $\Ahatzeroprime$ and $\Ahatoneprime$ are
$\Gammahat-$invariant, that is
$\gammahat(\Ahatzeroprime)=\Ahatzeroprime$ and
$\gammahat(\Ahatoneprime)=\Ahatoneprime$ for any $\gammahat \in
\Gammahat.$ Since $\Ahatzeroprime = \cup_{k \in \Integers} \,\,
\gammahat_0^k(\hat{Q}'_0),$ we can define $\Pmaphate$ on each
piece $\gammahat_0^k(\hat{Q}'_0)$ as the conjugated map
$$\gammahat \circ \Pmaphate \circ \gammahat^{-1} \,\, : \,\,
\gammahat_0^k(\hat{Q}'_0) \, \longrightarrow \, \Ahatoneprime.$$
By lemma \ref{lemma about lifted Pmap}, for two group elements
$\gammahat_1$ and $\gammahat_2 \in \Gammahat,$ the two maps
$\gammahat_1 \circ \Pmaphate \circ \gammahat^{-1}_1$ and
$\gammahat_2 \circ \Pmaphate \circ \gammahat^{-1}_2$ agree on the
intersection $\gammahat_1(\hat{Q}'_0) \cap
\gammahat_2(\hat{Q}'_0)$ whenever it is nonempty.
\end{proof}

\begin{corollary} \label{Corollary Pmaphat descend}
The transformation $\Pmaphate : \Ahatzeroprime \to \Ahatoneprime$
associated with the foliation $\Folehat$ descends to a smooth
Poincar\'e map $\Pmape : \Azeroprimeq \to \Aoneprimeq$ for the
foliation $\Fole$ under the covering bundle map $\Pi : \CtimesS
\to E$. In other words, for any $\hat{z} = (z,q_0) \in
\Ahatzeroprime$ the relation $\Pi_0 \circ \Pmaphate(\hat{z}) =
\Pmape \circ \Pi_0(\hat{z})$ holds.
\end{corollary}

\begin{proof}
The statement follows directly from corollary \ref{Corollary
Pmaphat extension}.
\end{proof}

On a side note, but still worth mentioning is a fact that follows
from the constructions in the proof of lemma \ref{lemma about
lifted Pmap}. It is not difficult to see that the Poincar\'e map
is not sensitive to (local) homotopies of the base loop
$\delta_0$. In other words, if $\delta_0$ is homotopic on $S_1$ to
another loop $\delta'_0$ passing through $q_0,$ then the two maps
obtained by the lifting of $\delta_0$ and $\delta'_0$ onto the
leaves of the foliation $\Folehat$ under the projection $pr_{S_1}$
will be equal, as long as $\delta'_0$ is close enough to
$\delta_0$ on $S_1$ or the radius $r$ is kept small enough. Thus,
if we slightly wiggle $\delta_0$ on $S_1$ but keep the base point
$q_0$ fixed, the resulting Poincar\'e map will stay the same.
Consequently, the same is true for $\Pmape.$

\section{Complex structure on the cross-section} \label{Section_complex_structure}

Apart from the smooth structure of a fiber bundle, the space $E$,
being a subset of $\CC,$ has a complex structure with respect to
which the foliation $\Fole$ is holomorphic and depends
analytically on the parameter $\e$. This fact provides the
foliation with very specific properties. On the other hand, the
Poincar\'e map $\Pmape : A'_0(q_0) \to A'_1(q_0)$ associated with
$\Fole$ captures some topological properties of the foliation.
Since some of those properties are strongly related to the
holomorphic nature of the foliation, we would like our Poincar\'e
map to reflect the complex analyticity of $\Fole$. So far $\Pmape$
is defined as a smooth map on a subdomain of the smooth surface
$A'_1(q_0)$ and therefore our next step is to induce a complex
structure on $A'_1(q_0)$ in which the Poincar\'e transformation is
holomorphic.

Since the closure of $A'_1(q_0)$ is transverse to $\Fole$, there
is an open neighborhood of $A'_1(q_0)$ in $B_{q_0}$ transverse to
$\Fole.$ Fix $\e \in D_r(0).$ Take a point $q' \in \Aoneprimeq$
and a complex cross-section $T_{q'}$ through $q',$ transverse to
$\Fole.$ More precisely, $T_{q'}$ is a complex segment, i.e. it
lies on a complex line through $q'$ and is a real two dimensional
disc.

The fact that the foliation $\Fole$ is holomorphic and
$\Aoneprimeq$ is smoothly embedded surface transverse to $\Fol$
provides us with convenient holomorphic flow-box charts of $\CC$.
A chart of this kind consists of an open neighborhood $FB(q')
\subset E$ of $q'$ and a biholomorphic map $$\betaq \, : \, \Disc
\times \Disc \longrightarrow FB(q')$$ with the following
properties: \vspace{1.5mm}

\noindent 1. \, $\betaq(0,0)=q';$ \vspace{1.5mm}

\noindent 2. \, $\betaq(\Disc \times\{0\})=T_{q'};$ \vspace{1.5mm}

\noindent 3. \, $\betaq(\{\zeta\} \times \Disc)$ is a connected
component of the intersection of $FB(q')$  with the leaf
$\Leafe(\betaq(\zeta,0))$ through the point $\betaq(\zeta,0)$ for
any $\zeta \in \Disc;$ \vspace{1.5mm}

\noindent 4. \, The portion of $\Aoneprimeq$ passing through
$FB(q')$ looks like the graph of a smooth map $\alphaq : \Disc \to
\Disc$ in the chart $\Disc \times \Disc.$ In other words
\begin{align*}\betaq^{-1}(FB(q') \cap \Aoneprimeq)=
\{(\zeta,\alphaq(\zeta)) \in \Disc \times \Disc \,\, | \,\,
\alphaq : \Disc \to \Disc \,\, \text{is smooth}\}.
\end{align*}

\noindent Denote by $\Uq$ the open subset $FB(q_0) \cap
\Aoneprimeq$ of $\Aoneprimeq.$ Let $pr_j : \Disc \times \Disc \to
\Disc$ be $pr_j(\zeta_1,\zeta_2)=\zeta_j,$ where $j=1,2.$ Define
the diffeomorphism
\begin{align*}
\phi_{q',\e} \,\, &: \,\, \Uq \, \longrightarrow \, \Disc \,\,\, \text{by} \\
\phi_{q',\e} \,\, &: \,\, q \longmapsto pr_1 \circ (\betaq^{-1})\big{|}_{\Uq}(q) \\
\phi_{q',\e}^{-1} \,\, &: \,\, \zeta \longmapsto
\betaq(\zeta,\alphaq(\zeta)).
\end{align*}
Consider the family of pairs
$\Atlas_{\e}(\Aoneprimeq)=\{(\Uq,\phi_{q',\e}) \,\, | \, \, q' \in
\Aoneprimeq\}.$

\begin{lemma} \label{lemma complex atlas}
The collection of charts $\Atlas_{\e}(\Aoneprimeq)$ is a
holomorphic atlas for the surface $\Aoneprimeq$ with charts
depending complex-analytically on $\e$.
\end{lemma}

\begin{proof}
The proof is a direct verification that the transition functions
$\phi_{q_1,\e} \circ \phi_{q_2,\e}^{-1}$ of two intersecting
charts are holomorphic in both the coordinate variable and the
parameter $\e$. It comes from the fact that the charts are
basically projections of the open patches $U_{q'} \subset
\Aoneprimeq$ onto the holomorphic cross-sections $T_{q'}$ along
the leaves of the foliation $\Fole$. Therefor the transition
transformations are going to be maps from one holomorphic
cross-section to another following locally the leaves of $\Fole$.
As the leaves depend holomorphically on the initial condition and
the parameter $\e$, we obtain the desired result. For more
details, one can look at \cite{ND1} or \cite{ND2}.
\end{proof}

The choice of complex structure on the surface $\Aoneprimeq$ is
justified by the next lemma. As it turns out, the map $\Pmape$ is
holomorphic in the complex structure $\Atlas_{\e}(\Aoneprimeq).$

\begin{lemma} \label{Lemma Complex Pmap}
The Poincar\'e map $\Pmape : \Azeroprimeq \to \Aoneprimeq$ from
corollary \ref{Corollary Pmaphat descend} associated to the
foliation $\Fole$ is holomorphic in the complex structure defined
by the atlas $\Atlas_{\e}(\Aoneprimeq)$ and depends
complex-analytically on the parameter $\e$.
\end{lemma}

\begin{proof}
The idea is based on the heuristic arguments of the previous lemma
\ref{lemma complex atlas}. Restricted to an open patch $U_{q_1}
\subset \Azeroprimeq$, the map $\Pmape$ sends $U_{q_1}$ inside
another open patch $U_{q_2}$. If we look at $\Pmape$ in the two
corresponding coordinate charts, we obtain a map form the
holomorphic cross-section $T_{q_1}$ to the holomorphic
cross-section $T_{q_2}$  by following the "lifts" of the loop
$\delta_0$ on the leaves of the foliation $\Fole$. The leaves
depend holomorphically on the initial condition and the parameter
$\e$, as well as both $T_{q_1}$ and $T_{q_2}$ are complex
segments, so we conclude that the map is as desired. More details
can be found in \cite{ND1} or \cite{ND2}.
\end{proof}

\begin{corollary}\label{Corollary Pmaphat complex structure}
The surface $\Ahatoneprime \subset \cc \times \{q_0\}$ has a
complex atlas
$$\Atlas_{\e}(\Ahatoneprime) = \{
(\Uhatzzero,\Phihatzzero) \, \, : \, \, \hat{z}_0 \in
\Ahatoneprime\},$$ such that the covering map $\Pi_0 :
\Ahatoneprime \to \Aoneprimeq$ is holomorphic with respect to the
complex atlas $\Atlas_{\e}(\Aoneprimeq)$. The new atlas makes the
lifted Poincar\'e map $\Pmaphate$ holomorphic, depending
complex-analytically on $\e.$
\end{corollary}

\begin{proof}
Since as a smooth covering map $\Pi_0$ is a local diffeomorphism,
simply pull back the complex structure given by
$\Atlas_{\e}(\Aoneprimeq)$ to the surface $\Ahatoneprime$.
\end{proof}

\section{Periodic orbits and complex cycles}

We proceed with the study of the Poincar\'e maps $\Pmape$ and
$\Pmaphate.$ More precisely, we are interested in the relationship
between their periodic orbits and the marked complex cycles of the
perturbed foliation $\Fole.$

\begin{lemma} \label{Lemma_periodic_orbits_cycles_lifted}
Let $r>0$ be the radius obtained in lemma \ref{lemma about lifted
Pmap} and let $\e \in D_{r}(0)$ be fixed. Then, the following
statements are true: \vspace{1.5mm}

\noindent {\em 1.} If $\Pmaphate : \Ahatzeroprime \to
\Ahatoneprime$ has an $m-$periodic orbit $\zhat_1,..., \zhat_m$ in
$\Ahatzeroprime$, then $\Pmape : \Azeroprimeq \to \Aoneprimeq$ has
an $m-$periodic orbit $q_1,..., q_m$ in $\Azeroprimeq$, where
$\Pi_0(\zhat_j) = q_j$ for $j = 1,..., m$.\vspace{1.5mm}

\noindent {\em 2.} Moreover, for each $j = 1,..., m$ the foliation
$\Fole$ has a marked complex cycle $(\Delta_j, q_j)$ with an
$m-$fold vertical representative $\delta_j$ contained in $E'_1 =
H^{-1}(A'_1)$. \vspace{1.5mm}

\noindent {\em 3.} Finally, for $j = 1,...,m$ the cycle
$(\Delta_j, q_j)$ is $m-$fold vertical, i.e. every representative
of the cycle is free homotopic in $E$ to $\delta_0^m$.
\end{lemma}

\begin{proof}
We start with the proof of the first part. Recall that
$$\Pi_0 = \Pi|_{\Ahatoneprime} \,\, : \,\, \Ahatoneprime \, \longrightarrow \, \Aoneprimeq$$ is
the universal covering map from the band $\Ahatoneprime$ to the
cylinder $\Aoneprimeq$ embedded in $E \subset \CC$. Let $q_j =
\Pi_0(\zhat_j) \in \Aoneprimeq$ for $j=1,..,m$. Due to the
conjugacy relation $\Pmape \circ \Pi_0 = \Pi_0 \circ \Pmaphate$,
the image $\{q_j\}_{j=1}^{m}$ of the $m-$periodic orbit
$\{\zhat_j\}_{j=1}^{m}$ of the map $\Pmaphate : \Ahatzeroprime \to
\Ahatoneprime$ is a periodic orbit of $\Pmape : \Azeroprimeq \to
\Aoneprimeq$ with possibly a smaller period. Clearly,
$\Pmape^m(q_1)=\Pmape^m(\Pi_0(\zhat_1))=\Pi_0 \circ
\Pmaphate^m(\zhat_1)=\Pi_0(\zhat_1)=q_1.$

Assume there exists a smaller $k < m$ such that $q_1=q_{k+1}$.
Since $\Pi_0 : \Ahatoneprime \to \Aoneprimeq$ is a covering map
and $\Gammahat$ restricted to $\Ahatoneprime$ is the deck group of
$\Pi_0$, there exists $\gammahat \in \Gammahat$ such that
$\zhat_{k+1}=\gammahat(\zhat_1)$. On the other hand,
$\zhat_{k+1}=\Pmaphate^k(\zhat_1).$ Thus,
$\Pmaphate^k(\zhat_1)=\gammahat(\zhat_1).$ Applying $\Pmaphate^k$
to the last equality we obtain
\begin{align*}
 \Pmaphate^{2k}(\zhat_1)&=\Pmaphate^k \circ
 \gammahat(\zhat_1)\\
 &=\gammahat \circ \Pmaphate^k(\zhat_1)\\
 &=\gammahat^2(\zhat_1).
\end{align*}
In general, $\Pmaphate^{jk}(\zhat_1)=\gammahat^j(\zhat_1)$ for any
$j \in \Naturals$. In particular, when $j=m$ we have
$\zhat_1=\Pmaphate^{mk}(\zhat_1)=\gammahat^m(\zhat_1)=\gammahat^m(\zhat_1).$
The identity $\zhat_1=\gammahat^m(\zhat_1)$ implies that
$\gammahat^m$ has a fixed point in $\Ahatoneprime$. As a deck
group of the universal covering, $\Gammahat \cong \Integers$ acts
freely on $\Ahatoneprime$, so $\gammahat^m=id_{\Ahatoneprime}$ and
since $\Gammahat$ has no torsion, $\gammahat=id_{\Ahatoneprime}.$
We reach the conclusion $\zhat_{k+1}=\zhat_1$ which is not true.
This concludes the proof of part one.

Next, we show that the second part holds. For convenience, let
$\zhat_{m+1}=\zhat_1.$ Observe that since all the points
$\zhat_1,..., \zhat_m$ belong to the same orbit of $\Pmaphate$,
they lie on the same leaf $\Leafhate(\zhat_1)$ from the foliation
$\Fole.$ Let $\deltahat_{\e}(\zhat_j,\zhat_{j+1}),$ for
$j=1,...,m,$ be the lift of $\delta_0$ on the leaf
$\Leafhate(\zhat_1)$ so that the path
$\deltahat_{\e}(\zhat_j,\zhat_{j+1})$ covers $\delta_0$ under the
projection $pr_{S_1}$ and connects the points $\zhat_j$ and
$\zhat_{j+1}.$ By the construction of the map $\Pmaphate$ in the
proof of lemma \ref{lemma about lifted Pmap}, all paths
$\deltahat_{\e}(\zhat_j,\zhat_{j+1})$ are contained in
$\pi^{-1}(A'_1) \times S_1.$ Therefore, the path $\deltahat_{\e} =
\cup_{j=1}^{m} \deltahat_{\e}(\zhat_j,\zhat_{j+1})$ is contained
in $\pi^{-1}(A'_1) \times S_1$ and  goes through the points
$\zhat_1,...,\zhat_m.$ Moreover, if we parametrize
$\deltahat_{\e}$ from $\zhat_1$ to $\zhat_m$ in a direction
induced by the orientation of $\delta_0$, its two endpoints are
$\zhat_1$ and $\zhat_{m+1}=\zhat_1$, so in fact $\deltahat_{\e}$
is a loop on the leaf $\Leafhate(\zhat_1)$.

For each $j=1,...,m$ denote by $\deltahat_j$ the closed curve
$\deltahat_{\e}$ but assuming that it starts form the point
$\zhat_j$. As point-sets all loops $\deltahat_j$ are the same
$\deltahat_{\e}$. The only difference between them is that the
parametrization for each $\deltahat_j$ starts from a different
point from the periodic orbit of $\Pmaphate$.

Fix some $j = 1,..., m$. When mapping $\deltahat_j$ with $\Pi$
back onto $E$ we obtain a loop $\delta_j = \Pi(\deltahat_j)$ lying
on the leaf $\Leafe(\Pi(\zhat_j)) =\Pi(\Leafhate(\zhat_j))$ from
the foliation $\Fol.$ Moreover, $\delta_j$ is contained in
$E'_1=\Pi(\pi^{-1}(A'_1) \times S_1).$ As discussed in \cite{PL55}
and \cite{PL57}, the loop $\delta_j$ is non trivial on
$\Leafhate(\zhat_j)$ and defines a marked complex cycle
$(\Delta_j,q_j).$

Let $Pr_{S_1} : \CtimesS \to \{0\}\times S_1$ be the map
$Pr_{S_1}(z,p) = (0,p)$. Having in mind that $\cc$ is
contractible, $Pr_{S_1}$ is a deformation retraction, so clearly
each $\deltahat_j$ is free homotopic in $\CtimesS$ to $\{0\}
\times \delta_0^m$. As $\Pi(\CtimesS) = E$, the map $\Pi$ sends
this free homotopy to a free homotopy in $E$ between $\delta_j =
\Pi(\deltahat_j)$ and $\delta_0^m = \Pi(0,\delta_0^m)$. Hence all
loops $\delta_j$ are $m-$fold vertical.

The third point of the current theorem follows directly from
definition \ref{Definition_delta_m_fold_vertical} and the
discussion in the paragraph right after it.
\end{proof}

\begin{lemma}\label{Lemma peridoic orbits and families and connection with cycles}
Let $\e'$ belong to the parameter disc $D_r(0),$ where the radius
$r>0$ is chosen as in lemma \ref{lemma about lifted Pmap}.
\vspace{1mm}

\noindent{1.} If $\Phat_{\e'} : \Ahatzeroprime \to \Ahatoneprime$
has an isolated $m-$periodic orbit $\{\zhat_j\}_{j=1}^{m}$
contained in $\Ahatzeroprime$ then $P_{\e'} : \Azeroprimeq \to
\Aoneprimeq$ has an isolated $m-$periodic orbit
$\{q_j\}_{j=1}^{m}$ in $\Azeroprimeq$, where $\Pi_0(\zhat_j) =
q_j$ for $j = 1,..., m$.\vspace{1mm}

\noindent{2.} Moreover, there exists a disk $D_{r'}(\e') \subset
D_r(0)$ with a  small radius $r^{\prime}>0$ such that for any
embedded in $D_{r'}(\e')$ curve $\eta',$ passing through $\e',$
there exists a continuous family
$\bigl(\zhat_1(\e),...,\zhat_m(\e)\bigr)_{\e \in \eta'}$ of
periodic orbits for the map $\Pmaphate$ which for $\e=\e'$ becomes
$\zhat_1,...,\zhat_m.$ Furthermore, this continuous family is
mapped by $\Pi_0$ to a continuous family of periodic orbits
$(q_1(\e),...,q_m(\e))_{\e \in \eta'}$ for the transformation
$\Pmape$ which for $\e=\e'$ becomes the orbit $q_1,...,q_m.$
\vspace{1mm}

\noindent{3.} If $\Pmaphate$ has a continuous family of periodic
orbits $\{\zhat_j(\e)\}_{j=1}^{m}$ on $\Ahatzeroprime$ for $\e$
varying on some curve $\etatilde$ embedded in $D_r(0),$ then the
perturbed foliation $\Fole$ has continuous families of marked
cycles $\{(\Delta_j(\e),q_j(\e))\}_{\e \in \etatilde}$ for
$j=1,...,m$ where $q_j(\e) = \Pi_0(\zhat_{\e})$. 
\end{lemma}

\begin{proof}
Point one follows from lemma
\ref{Lemma_periodic_orbits_cycles_lifted} together with the
conjugacy condition $\Pi_0 \circ \Phat_{\e'} = P_{\e'}
\circ\Pi_0$. As $\Pi_0$ is locally a diffeomorphism, then the fact
that $\{\zhat_j\}_{j=2}^{m}$ is isolated implies that
$\{q_j\}_{j=1}^{m}$ is isolated too.

As $\Phat_{\e'}^m(\zhat_1)=\zhat_1,$ we choose a chart
$(\Uhatzone,\Phihatzone)$ form the atlas
$\Atlas_{\e}(\Ahatoneprime)$ around the point $\zhat_1$ and a
smaller neighborhood $\Uhatzone'$ of the same point such that
$\Uhatzone^{\prime} \subset \Uhatzone$ and
$\Phat_{\e'}^m(\Uhatzone^{\prime}) \subset \Uhatzone.$ Let
$D^{\prime}=\Phihatzoneprime(\Uhatzone^{\prime}) \subset \Disc$
where $\Phihatzoneprime(\zhat_1)=0 \in D^{\prime}.$ If
$r^{\prime}>0$ is chosen small enough, then $$\Phat^{(m)}_{\e} =
\Phihatzone \circ \Pmaphate^m \circ \Phihatzone^{-1} \, : \,
D^{\prime} \, \longrightarrow \, \Disc$$ for $\e \in
D_{r^{\prime}}(\e') \subset D_r(0).$ Notice that
$\Phat^{(m)}_{\e'}(0)=0.$ The complex valued function
$$\tilde{F} : D^{\prime} \times D_{r'}(\e') \to \cc \,\,\, \text{defined as}\,\,\,
\tilde{F}(\zeta,\e) = \Phat^{(m)}_{\e}(\zeta)-\zeta$$ is
holomorphic with respect to $\zeta \in D^{\prime}$ and with
respect to $\e \in D_{r^{\prime}}(\e').$  By Hartogs' theorem
\cite{GR}, $\tilde{F}$ is holomorphic with respect to $(\zeta,\e)
\in D^{\prime} \times D_{r^{\prime}}(\e').$ Since
$\Phat^{(m)}_{\e'}(0)=0,$ the point $(0,\e')$ is a zero of $\tilde
F,$ that is $\tilde F(0,\e')=0.$

Let us look at the zero locus of $\tilde F$ in $D^{\prime} \times
D_{r^{\prime}}(\e').$ The fact that the periodic orbit
$\{\zhat_j\}_{j=1}^{m}$ is isolated means that $\zhat_1$ is an
isolated fixed point for the map $\Phat_{\e'}^m$. Therefore $0$ is
an isolated fixed point for $\Phat^{(m)}_{\e'}$ and thus, it is an
isolated zero for the holomorphic function $\Ftilde(\zeta,\e')$
regarded as a function of $\zeta$ only. By Weierstrass'
preparation theorem \cite{GR}, \cite{Ch}, we can write
$$\tilde F(\zeta,\e)=\prod_{j=1}^{s}(\zeta -
\alpha_j(\e))\theta(\zeta,\e),$$ where $\theta(0,\e') \neq 0$ and
$\{\alpha_j(\e) \, : \, j=1,...,s\}$ depend analytically on $\e
\in D_{r'}(\e'),$ satisfying the conditions
$\alpha_1(\e')=...=\alpha_s(\e')=0$ and possibly branching into
each other.

Now, let $\eta'$ be some simple curve embedded in the disc
$D_{r^{\prime}}(\e')$ and passing through $\e'.$ For $\e$
 varying on $\eta'$, we can choose a branch, denoted for simplicity by
$\alpha_1(\e)$. Then the desired continuous family for $\Pmaphate$
can be constructed by setting
$\zhat_1(\e)=\Phihatzone^{-1}(\alpha_1(\e))$ and
$\zhat_{j+1}(\e)=\Pmaphat^{j}(\zhat_1(\e))$ for $j=1,...,m-1.$ Its
image under the covering $\Pi_0$ will provide the continuous
family of periodic orbits $\{q_j(\e)\}_{j=1}^m$ for $\Pmape.$

The third point of the current statement follows directly from
lemma \ref{Lemma_periodic_orbits_cycles_lifted} part two. For each
$\e \in \eta'$ and $j = 1,..., m$ the constructed representative
$\delta_j(\e) \subset E'_1$ depends in fact continuously on the
parameter $\e \in \eta'$ because the leaves of all foliations we
work with depend continuously on the initial points and the
parameter $\e$. Thus the loops $\delta_j(\e)$ and the base points
$q_j(\e)$, continuously depending on $\e$, give rise to continuous
families $\{(\Delta_j(\e), q_j(\e))\}_{\e \in \eta'}$
\end{proof}

\section{Proof of theorem \ref{The_Main_Theorem}}

From now on we assume that $a=a_m$ as provided by lemma \ref{Lemma
Example Periodic Orbit Existence}. Let $\eta$ be an embedded in
$D_r(0)$ curve, connecting $\e_0$ to $0$. For convenience, define
a natural linear order $\preceq$ on it so that $0 \prec \e_0.$ We
begin this section with a summary of the proof.

First, with the help of lemma \ref{Lemma Example Periodic Orbit
Existence}, we show that the Poincar\'e map $\Pmap$ of the
foliation $\Fol$ has an isolated periodic orbit
$\{q_j\}_{j=1}^{m}$ for the fixed parameter $a=a_m$ and $\e_0$
very close to $\e_m \in D_r(0)$. Then, $\{q_j\}_{j=1}^{m}$ can be
extended to a continuous family of periodic orbits
$\{q_j(\e)\}_{j=1}^{m}$ on $\Azeroprimeq$ for values of the
parameter $\e$ defined on an relatively open subset $\etamax$ of
the path $\eta$. We also find out that there exists $\e^* \in
\etamax$ such that if $\e \in \etamax$ and $\e \prec \e^*$ then
some $q_{j_0}(\e) \in \Azeroprimeq \setminus \Azeroq$. We denote
$q_{j_0}(\e) = q(\e)$. By point 3 of lemma \ref{Lemma peridoic
orbits and families and connection with cycles}, there exists a
continuous family of marked cycles $\{(\Delta(\e),q(\e))\}_{\e \in
\etamax}$ defined on $\etamax$. For the value $\e=\e_0$ the cycle
$(\Delta(\e_0),q(\e_0))$ is limit $m-$fold vertical and has a
representative $\delta(\e_0)$ contained in the domain
$E_0=H^{-1}(A_0)$. As $\e$ moves on $\etamax$ towards $0$, it
passes through the point $\e^*$ and as a result of this the point
$q(\e)$ leaves $\Azeroq$ and therefore it leaves $E_0$ as well.
Consequently, for any $\e \in \etamax$ with $\e \prec \e^*$ no
representative of $(\Delta(\e),q(\e))$ is contained in $E_0$
because all of them pass through the same base point $q(\e)$ and
$q(\e)$ is not in $E_0$ anymore. This last fact concludes the
proof of theorem \ref{The_Main_Theorem} with $\sigma = \etamax$.

Next, we continue with the detailed proof of the main result of
this article.
The Poincar\'e map constructed in section
\ref{Section_Local_Poincare_map} is in fact the local Poincar\'e
map from the introduction of the article, defined on the complex
cross-section $T'_{q_0}= \{(x,0) \in \CC \, | \,|x-1| < r'_0\}$
where $r'_0$ is very small. We denote this local map by
$P^{loc}_{\e} : T_{q_0} \to T'_{q_0}$. Notice that $P^{loc}_{\e}$
is constructed using the tubular neighborhood $N(\delta_0)$
together with a projection, we call $\varrho : N(\delta_0) \to
A(\delta_0)$, coming from the direct product structure on
$N(\delta_0)$ selected in section
\ref{Section_Local_Poincare_map}. Since the covering map $\Pi :
\CtimesS \to E$ is a local diffeomorphism, the tubular
neighborhood $N(\delta_0)$ inherits via $\Pi$ the product
structure of $\CtimesS$ together with a projection $\varrho'$
coming from the natural projection $pr_{S_1}$ which maps
$\CtimesS$ onto $S_1$. As already mentioned in the introduction,
there exists an isotopy on $N(\delta_0)$ that keeps $A(\delta_0)$
fixed point-wise and sends one product structure to the other
\cite{Hr}. Therefore, the lifts of $\delta_0$ via $\varrho$ and
$\varrho'$ on any near-by leaf of $\Fole$ will coincide as
point-sets. In conclusion, the map $P^{loc}_{\e}$ is in fact a
representation of the non-local Poincar\'e map $\Pmape$ in a
complex chart from the atlas $\Atlas_{\e}(\Aoneprimeq)$ defined in
section \ref{Section_complex_structure}. Together with this, by
corollary \ref{Corollary Pmaphat complex structure}, the local map
$P^{loc}_{\e}$ on the cross-section $T'_{q_0}$ is a representation
of the lifted Poincar\'e transformation $\Pmaphate$ in a complex
chart from the atlas $\Atlas_{\e}(\Ahatoneprime)$. Combining this
fact with lemma \ref{Lemma Example Periodic Orbit Existence}, we
conclude that for a choice of $\e_0$ near $\e_m \in D_{r}(0)$ the
transformation $\Pmaphate$ has an isolated $m-$ periodic orbit
$\{\zhat_j\}_{j=1}^{m}$. By lemma
\ref{Lemma_periodic_orbits_cycles_lifted}, for $j=1,...,m$ there
are $m-$fold vertical limit cycles $(\Delta_j,q_j)$ each having a
representative $\delta_j$ contained in $E'_1$. Comparing the
construction of $P^{loc}_{\e}$ with that of $\Pmaphate$ and
$\Pmape$ as well as having in mind the fact that the lifts are
independent on the choice of a product structure on $N(\delta_0)$,
one concludes that $\delta_j$ is contained in $N(\delta_0)$ which
is a thin neighborhood of $\delta_0$ and in its own turn is a
subset of $E_0$. Thus, the representative $\delta_j$ of the limit
cycle $(\Delta_j,q_j)$ is an $m-$fold vertical loop contained in
the domain $E_0$ for $j=1,..,m$.

By lemma \ref{Lemma peridoic orbits and families and connection
with cycles}, there exists $D_{r_0}(\e_0) \subset D_r(0)$ for some
$r_0>0,$ such that if $\eta_0=\eta \cap D_{r_0}(\e_0),$ then there
is a continuous family of periodic orbits
$(\zhat_1(\e),...,\zhat_m(\e))_{\e \in \eta_0}$ of the map
$\Pmaphate$ on the cross-section $\Ahatzeroprime.$

Define $\etamax \subseteq \eta$ as the maximal relatively open
subset of $\eta$ on which the continuous family
$(\zhat_1(\e),...,\zhat_m(\e))_{\e \in \etamax}$ of periodic
orbits for $\Pmaphate$ exists on the cross-section
$\Ahatzeroprime$. Since $\eta_0 \neq \varnothing$ is relatively
open in $\eta,$ the inclusion $\eta_0 \subseteq \etamax$ holds and
therefore $\etamax \neq \varnothing.$ By point 3 from lemma
\ref{Lemma peridoic orbits and families and connection with
cycles}, for each $j =1,..,m$ there exists a continuous family of
marked complex cycles $\{(\Delta_j(\e),q_j(\e))\}_{\e \in
\etamax}$ with $q_j(\e)=\Pi(\zhat_j(\e)).$ Near $\e_0 \in \etamax$
the cycles $(\Delta_j(\e),q_j(\e))$ are limit and have $m-$fold
vertical representatives $\delta_j(\e)$ contained in $E_0$ because
when $\e=\e_0$ each cycle $(\Delta_j(\e_0),q_j(\e_0)) =
(\Delta_j,q_j)$ is limit and has an $m-$fold vertical
representative, namely $\delta_j=\delta_j(\e_0),$ contained inside
the domain $E_0.$ We would like to find out what happens to the
cycles as $\e$ varies on $\etamax.$

Let $\eta'$ be the set of all $\e$ from $\etamax$ for which the
periodic orbits from the continuous family
$(\zhat_1(\e),...,\zhat_m(\e))_{\e \in \etamax}$ are entirely
contained in $\Ahatzero \subset \Ahatzeroprime$. As we already
saw, at $\e_0$ the orbit $\zhat_1(\e_0),...,\zhat_m(\e_0)$ is
inside $\Ahatzero$ and by continuity, the orbits
$\zhat_1(\e),...,\zhat_m(\e)$ are also contained in $\Ahatzero$
for $\e$ near $\e_0.$ This fact shows that $\eta' \neq
\varnothing$ and in fact it has a nonempty interior.

Let $\etwostar =\inf_{\eta}(\etamax)$ be the infimum of $\etamax$
with respect to the linear ordering on $\eta$. Then,
$D_{\frac{1}{N}}(\e^*) \cap \etamax \neq \varnothing$ for all $N
\in \Naturals.$ Similarly, define $\e^*=\inf_{\eta}(\eta')$ as the
infimum of $\eta'.$ The inclusion $\eta' \subseteq \etamax$
implies that $\etwostar \preceq \e^*.$ We are going to show that
$\etwostar \neq \e^*$.

Assume $\etwostar=\e^*,$ that is for all $N \in \Naturals$ there
exists $\e_N \in D_{\frac{1}{N}}(\etwostar) \cap \etamax$ such
that $\zhat_1(\e_N),...,\zhat_m(\e_N)$ is contained in
$\Ahatzero$. As explained in point 2 of lemma \ref{Lemma peridoic
orbits and families and connection with cycles} the family of
periodic orbits $(\zhat_1(\e),...,\zhat_m(\e))_{\e \in \etamax}$
is mapped by $\Pi_0$ to a periodic family
$(q_1(\e),...,q_m(\e))_{\e \in \etamax}$ of the map $\Pmape$ on
the surface $\Azeroprimeq.$ Moreover, the corresponding orbits
$q_1(\e_N),...,q_m(\e_N)$ are inside $\Azeroq \subset
\Azeroprimeq$ for all $N \in \Naturals.$ In particular, the
sequence $\{q_1(\e_N)\}_{N \in \Naturals}$ is contained in the
compact cylinder $\overline{\Azeroq}.$ Then, there exists $q_1^{*}
\in \overline{\Azeroq}$ and a subsequence $\{q_1(\e_n)\}_{n \in
\Naturals}$ such that $\lim_{n \to \infty} q_1(\e_n)=q_1^*$ and
$\lim_{n \to \infty} \e_n=\etwostar.$ By continuity, the identity
$P_{\e_n}^m(q_1(\e_n))=q_1(\e_n)$ converges to
$P^m_{\etwostar}(q_1^*)=q_1^*$ as $n \to \infty.$ Generate a
periodic orbit $q_1^*,...,q_m^*$ by setting
$q_{j+1}^*=P^j_{\etwostar}(q_1^*)$ for $j=1,...,m-1.$ Since
$q_{j+1}(\e_n)=P^j_{\e_n}(q_1(\e_n))$, the limit for each
$q_j(\e_n) \in \Azeroq$ is $q_j^* \in \overline\Azeroq$ as $n \to
\infty.$ Thus, the periodic orbit $q_1^*,...,q_m^*$ is the limit
of the periodic orbits $q_1(\e_n),...,q_m(\e_n)$ and is contained
in $\overline{\Azeroq}$.

We will show that under the current assumptions $\etwostar =0$.
Assume that $\e^{**} \neq 0$. Then $\{\e \in \eta \, | \, \e \prec
\etwostar\} \neq \varnothing.$ We proceed as in the proof of lemma
\ref{Lemma peridoic orbits and families and connection with
cycles}. The point $q_1^* \in \overline{\Azeroq}$ is fixed by the
map $P^m_{\etwostar}$. Take a complex chart $(U_{q_1^*},
\phi_{q^{*}_1,\etwostar})$ form the atlas
$\Atlas_{\etwostar}(\Aoneprimeq)$ around $q_1^*$ and a smaller
neighborhood $U'_{q_1^*} \subset U_{q_1^*}$ of the same point such
that $P^m_{\etwostar}(U'_{q_1^*}) \subset U_{q_1^*}$. Let
$D^{\prime}=\phi_{q^*_1,\etwostar}(U'_{q_1^*,\etwostar}) \subset
\Disc$ where $\phi_{q^{*}_1,\etwostar}(q_1^*)=0 \in D^{\prime}$.
Choose $r^*>0$ small enough such that
$$P^{(m)}_{\e} = \phi_{q^*_1,\e} \circ \Pmape^m \circ \phi_{q^{*}_1,\e}^{-1}
\, : \, D^{\prime} \, \longrightarrow \Disc$$ for $\e \in
D_{r^*}(\etwostar) \subset D_r(0).$ Notice that
$P^{(m)}_{\etwostar}(0)=0.$ The complex valued function
$$\tilde{F} : D^{\prime} \to \cc \,\,\, \text{defined as}\,\,\,
\tilde{F}(\zeta,\e)=P^{(m)}_{\e}(\zeta)-\zeta$$ is holomorphic
with respect to $(\zeta,\e) \in D^{\prime} \times
D_{r^*}(\etwostar).$ Since $P^{(m)}_{\etwostar}(0)=0,$ the point
$(0,\etwostar)$ is a zero of $\tilde{F}$, that is
$\tilde{F}(0,\etwostar)=0$.

We are interested in the zero locus of $\tilde{F}$ in the domain
$D^{\prime} \times D_{r^*}(\etwostar)$. If we assume for a moment
that $\tilde{F}(\zeta,\e)\equiv 0$ on $D^{\prime}$ then we would
have the identity $P^{(m)}_{\e}(\zeta) \equiv \zeta$ on
$D^{\prime}$ and therefore $\Pmape^m(q)\equiv q$ on the open
subset $U'_{q^*_1} \subset \Azeroprimeq.$ Because of the
analyticity of $\Pmape^m(q)$ with respect to both $q$ and $\e,$
the identity $\Pmape^m(q)\equiv q$ will hold on all of
$\Azeroprimeq$ and for all $\e \in D_{r}(0)$. In particular, it
will be true for $\e=\e_0$. But for that value the map $\Pmape^m$
has an isolated fixed point $q_1(\e_0) \in \Azeroq \subset
\Azeroprimeq$ which leads to a contradiction. Therefore
$\tilde{F}$ is not identically zero.

There are two cases for $\tilde{F}.$ Either
$\tilde{F}(\zeta,\etwostar) \equiv 0$ or
$\tilde{F}(\zeta,\etwostar) \not\equiv 0$ for $\zeta \in
D^{\prime}$. For both of those options $\tilde{F}$ can be written
as
$$\tilde{F}(\zeta,\e)=(\e-\etwostar)^bF(\zeta,\e)$$
where $F(\zeta,\etwostar)\not\equiv 0$ a nd $b \geq 0.$ When $b
> 0$ we have the first case and when $b=0$ we have the second
case.

Let us look at the zero locus of $F$. By Weierstrass' preparation
theorem \cite{Ch}, \cite{GR}, $F$ can be written as
$$F(\zeta,\e)=\prod_{j=1}^{s}(\zeta - \alpha_j(\e))\theta(\zeta,\e),$$
where $\theta(0,\etwostar) \neq 0$ and $\{\alpha_j(\e) \, : \,
j=1,...,s\}$ depend analytically on $\e,$ satisfying the
equalities $\alpha_1(\etwostar)=...=\alpha_s(\etwostar)=0$ and
possibly branching into each other. Without loss of generality, we
can think that $D^{\prime}$ is chosen small enough so that
$\theta(\zeta,\e) \neq 0$ for all $(\zeta,\e) \in D^{\prime}\times
D_{r^*}(\etwostar).$ Let
$\tilde{\alpha}_j(\e)=\phi_{q^*_1,\e}^{-1}(\alpha_j(\e))$. Since
$q_1(\e_n) \to q_1^*,$ there exists $N_0 \in \Naturals$ such that
$q_1(\e_n) \in U'_{q^*_1}$ for $n>N_0.$ By the continuity of
$q_1(\e),$ for each $\e \in D_{r^*}(\etwostar) \cap \etamax$ we
have that $q_1(\e)=\tilde{\alpha}_j(\e)$ for some $j=1,..,s.$ For
simplicity of notation, assume $q_1(\e)=\tilde{\alpha}_1(\e)$.
Thus, $q_1(\e)$ converges to $q_1^*$ as $\e \to \etwostar$ always
staying on the zero locus of $F$. Therefore we can extend
$q_1(\e)$ continuously on $\eta$ past $\etwostar$ by setting
$q_1(\e)=\tilde{\alpha}_1(\e)$ for $\e \in D_{r^*}(\etwostar) \cap
\{\e \in \eta \, : \, \e \preceq \etwostar\}.$ By construction,
the identity $\Pmape^m(\tilde{\alpha}_1(\e))=
\tilde{\alpha}_1(\e)$ holds and if we set
$q_{j+1}(\e)=\Pmape^j(\tilde{\alpha}_1(\e))$ we obtain a
continuation of the family $q_1(\e),...,q_m(\e)$ on the relatively
open arc $D_{r^*}(\etwostar) \cap \{\e \in \eta \, : \, \e \preceq
\etwostar\}$. As a result we have a continuous family
$(q_1(\e),...,q_m(\e))_{\e \in \tilde\eta}$ of periodic orbits of
$\Pmape$ defined for $\e \in \tilde\eta = (D_{r^*}(\etwostar) \cap
\{\e \in \eta \, : \, \e \preceq \etwostar\}) \cup \etamax$ which
is relatively open in $\eta$.

Since the family $\zhat_1(\e),...,\zhat_m(\e)$ is the lift of
$q_1(\e),...,q_m(\e)$ for $\e \in \etamax$ and the latter extends
on $\tilde\eta \supset \etamax,$ the former also extends on
$\etatilde$ as a family of periodic orbits for $\Pmaphate$ on the
cross-section $\Ahatzeroprime.$ This conclusion contradicts the
maximality of $\etamax$, stemming from the assumption that
$\etwostar \neq 0.$ Therefore $\etwostar=0$ and
$q_1(0),...,q_m(0)$ is a periodic orbit of
$P_{0}=id_{\Azeroprimeq}.$ For that reason,
$q_1(0)=...=q_m(0)=q^*$ inside $\Azeroprimeq$.

Take a complex chart $(U_{q^*},\phi_{q^*,0})$ on $\Aoneprimeq$
around the point $q^*$ and choose a smaller neighborhood
$U_{q^*}^{\prime} \subset U_{q^*}$ of $q^*$ such that
$\Pmape^k(U_{q^*}^{\prime}) \subset U_{q^*}$ for all $k=1,...,m$
and $\e \in D_{\tilde r_0}(0),$ where $\tilde r_0>$ is small
enough. Let $D^{\prime}=\phi_{q^*,0}(U_{q^*}^{\prime}) \subset
\Disc$ and
$$P_{q^*,\e}=\phi_{q^*,\e} \circ \Pmape \circ \phi_{q^*,\e}^{-1}
\, : \, D^{\prime} \longrightarrow \Disc.$$ Denote by
$\zeta_j(\e)=\phi_{q^*,\e}(q_j(\e))$ for $\e \in D_{\tilde r_0}(0)
\cap \etamax = \eta_0$ and $j=1,...,m$. Then
$\zeta_1(\e),...,\zeta_m(\e)$ is a periodic orbit for $P_{q^*,\e}$
in $D^{\prime}$. Notice,that due to the holomorphic nature of the
map $\Pmape,$ those $\e \in \etamax$ for which $q_i(\e)=q_j(\e),$
where $1\leq i < j \leq m,$ are isolated because the family at
$\e_0$ consists of an isolated $m-$periodic orbit. As before,
$P_{q^*,\e}(\zeta)$ is holomorphic with respect to $(\zeta,\e)$.
Then we can write it as $$P_{q^*,\e}(\zeta)=\zeta + \e^l I(\zeta)
+ \e^{l+1}R(\zeta,\e)$$ where $I(\zeta) \not\equiv 0$ and $l \geq
1.$ If we iterate this map $m$ times we obtain the representation
$$P_{q^*,\e}^m(\zeta)=\zeta + \e^l m I(\zeta) + \e^{l+1}R_{(m)}(\zeta,\e).$$
For $\e \in \eta_0 \setminus \{0\}$ the equations
\begin{align*}
P_{q^*,\e}(\zeta)-\zeta &= \e^l(I(\zeta) + \e R(\zeta,\e))=0
\,\,\,\,
\text{and}\\
P_{q^*,\e}^m(\zeta)-\zeta &= \e^l(m I(\zeta) + \e
R_{(m)}(\zeta,\e))=0
\end{align*}
are divisible by $\e^l$ and thus, become
\begin{equation}\label{equation for periodic orbits}
I(\zeta) + \e R(\zeta,\e)=0 \,\,\,\,\, \text{and} \,\,\,\,\, m
I(\zeta) + \e R_{(m)}(\zeta,\e)=0
\end{equation}
The function $I(\zeta)$ is not identically zero, so it has
isolated zeroes. Choose $D^{\prime \prime} \subset D^{\prime}$ to
be a small closed disc centered at zero, so that no zeroes of
$I(\zeta)$ are contained in $D^{\prime \prime} \setminus \{0\}.$
In particular, $I(\zeta) \neq 0$ for $\zeta \in
\partial D^{\prime \prime}.$ We can decrease the parameter radius
$r_0>0$ enough so that by Rouche's theorem \cite{Smth} the
equations (\ref{equation for periodic orbits}) will have the same
number of zeroes, counting multiplicities, as the equation
$I(\zeta)=0$. Clearly, all zeroes of $P_{q^*,\e}(\zeta)-\zeta$ are
zeroes of $P_{q^*,\e}^m(\zeta)-\zeta$ because the fixed points of
$P_{q^*,\e}$ are fixed points of $P_{q^*,\e}^m$ but not the other
way around. On the other hand, as already noted, for almost every
$\e \in D_{r_0}(0)$ there is an $m-$periodic orbit
$\zeta_1(\e),...,\zeta_m(\e)$ for the map $P_{q^*,\e}$ inside
$D^{\prime \prime}.$ Thus, we can see that
$P_{q^*,\e}^m(\zeta)-\zeta$ has at least $m$ zeroes more than
$P_{q^*,\e}(\zeta)-\zeta,$ which contradicts the fact that both of
these should have the same number of zeroes. The contradiction
comes from the assumption that $\etwostar=\e^*.$ Therefore we
conclude that $\etwostar \neq \e^*$ and in fact $\etwostar \prec
\e^*$.

Let $\eta_1 = \{\e \in \etamax \, | \, \etwostar \prec \e \prec
\e^*\}.$ Then for any $\e \in \eta_1$ at least one
$\zhat_{j_0}(\e)$ is contained in $\Ahatzeroprime$ but not in
$\Ahatzero$. Then, its image $q(\e) = \Pi_0(\zhat_{j_0}(\e))$
varies continuously on $\Azeroprimeq$ with respect to $\e \in
\etamax$. Moreover, when $\e \in \eta_1 \subset \etamax$ what
happens is that $q(\e) \in \Azeroprimeq \setminus \Azeroq$. As the
set $ \Azeroprimeq \setminus \Azeroq$ is disjoined from the domain
$E_0$, the point $q(\e)$ is located outside of $E_0$. Point 3 of
lemma \ref{Lemma peridoic orbits and families and connection with
cycles} guarantees the existence of a continuous family of marked
cycles $\{(\Delta(\e),q(\e))\}_{\e \in \etamax}$ defined on
$\etamax$. For $\e=\e_0$ the cycle $(\Delta(\e_0),q(\e_0))$ is
limit $m-$fold vertical and has a representative $\delta(\e_0)$
contained in the domain $E_0$. As $\e$ moves on $\etamax$ towards
$0$, it passes through the point $\e^*$ and as a result the point
$q(\e)$ leaves $E_0$. Thus, for any $\e \in \eta_1 \subset
\etamax$ no representative of $(\Delta(\e),q(\e))$ is contained in
$E_0$ because all of them pass through the base point $q(\e)$
which is not in $E_0$ anymore. The proof of theorem
\ref{The_Main_Theorem} is completed with $\sigma = \etamax$.

\section{Concluding remarks}

The choice of the family (\ref{ExampleFoliation}) comes into play
mostly in the proof of the existence of multi-fold vertical cycles
for line fields of type (\ref{GeneralFoliation}). It is
specifically designed to facilitate the computation in the first
part of the article, where we study the bifurcation of periodic
orbits from a resonant parabolic fixed point of the Poincar\'e
map. Establishing the link between a foliation of type
(\ref{GeneralFoliation}) and the resonant terms in the normal form
of its corresponding Poincar\'e map seems hard. We can see that in
our simple example (\ref{ExampleFoliation}) we have quite involved
computations in order to show the non-triviality of the resonant
normal form of $\Pmap$. For the second part of the article, in
which we construct a non-local Poincar\'e transformation and we
the study the topological properties and rapid evolution of the
multi-fold limit cycles of $(\ref{ExampleFoliation})$, we do not
seem to need that much the explicit form of the foliation of type
(\ref{GeneralFoliation}). In fact the central role is played by
the polynomial $H$ and its geometric-topological properties. Since
$H$ is quite simple, so is its geometry and consequently the
topology of the fiber bundle $H : E \to B$. With some additional
modifications one could choose a more complicated polynomial $H$
and carry out similar constructions and prove rapid evolution for
more general families of type (\ref{GeneralFoliation}), provided
that the existence of a periodic orbit for the Poincar\'e map is
assumed. In \cite{ND1}, or alternatively in \cite{ND2}, the reader
could see (\ref{GeneralFoliation}) studied in a more general form.
In these works the approach and the general philosophy of the
current article are preserved, but the interplay between the
topology of the foliation and the dynamics of the Poincar\'e map
are quite more interesting. The map branches and its branching is
inherently related to the homotopy class of the loop $\delta_0$ on
the surface of a fixed fiber $S_{c_0}$ of $H$.

\section*{Acknowledgements}
I would like to express my sincere gratitude to Yulij Sergeevich
Ilyashenko for proposing this problem to me. His help, great
enthusiasm, useful remarks and clever observations have been
really invaluable. I am also very thankful to John Hamal Hubbard
who has been a source of knowledge and enthusiasm during my work
on this project.

\end{document}